\documentclass[10pt]{amsart}

\usepackage{amsmath, amssymb, amsthm}
\usepackage{graphicx}
\usepackage{multirow}
\usepackage{caption}
\usepackage{tabularx}
\usepackage{hyperref}  

\theoremstyle{plain}
\newtheorem{theorem}{Theorem}[section]
\newtheorem{lemma}[theorem]{Lemma}

\theoremstyle{definition}

\theoremstyle{remark}

\makeatletter
\def\subsection{\@startsection{subsection}{2}%
  \z@{.7\linespacing\@plus.7\linespacing}{.5\linespacing}%
  {\normalfont\normalsize\bfseries}} 
\makeatother

\begin{document}

\title[Undecidability of 5-polyomino translation tiling problem]
{Undecidability of tiling the plane with a set of \\ 5 polyominoes}

\author{Yoonhu Kim}
\address{\parbox[t]{\textwidth} {Yoonhu Kim \\ Gyeonggi Science High School \\ Gyeonggi-do 16297, Korea} }
\email{rumstis@gmail.com}


\begin{abstract}
In this paper, we give a proof that it is undecidable whether a set of five polyominoes can tile the plane by translation.
The proof involves a new method of labeling the edges of polyominoes,
making it possible to assign whether two edges can match for any set of two edges chosen.
This is achieved by dedicating 1 polyomino to the labeling process.
\end{abstract}

\maketitle

\section{Introduction}

Given a finite set of tiles, deciding whether they can tile the plane is an important problem.
One of the earliest researches on the decidability of tilings was the study on Wang tiles \cite{HW} by Hao Wang in 1961.
Wang tiles are unit squares with each edge given a color(Figure $1$ shows an example).

\begin{figure}[ht]
\centering
\includegraphics[width=0.25 \textwidth]{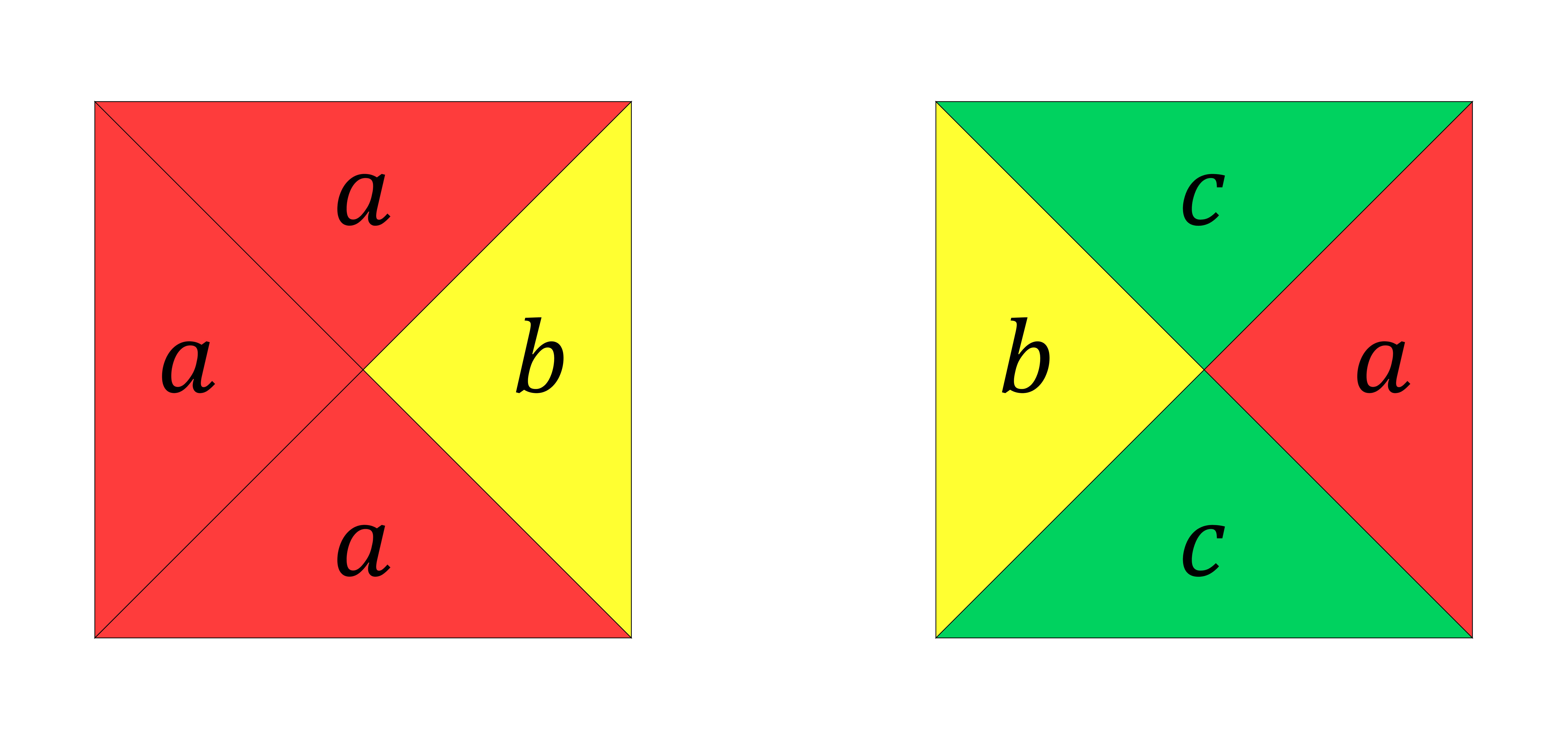}
\caption{A set of 2 Wang tiles}
\label{fig:Figure1}
\end{figure}

The problem known as the {\it domino problem} is the following: given a finite set of Wang tiles,
decide if they can tile the plane under the rule that matching edges must share the same color.
Berger proved that the domino problem is undecidable by simulating a Turing machine with Wang tiles \cite{RB}.
These Wang tiles can tile the plane if and only if the simulated Turing machine does not halt.

By reduction from the domino problem to tiling with a set of polyominoes, Golomb 
proved that it is undecidable whether a finite set of polyominoes can tile the plane \cite{SW}.
This is done by encoding the colors of the Wang tiles using unique dents and bumps on a large square polyomino.
This construction requires the same number of polyominoes as the number of Wang tiles.

Ollinger considered the problem of tiling the plane with $k$ polyominoes by translation, where $k$ is a fixed integer.
This is known as the $k$-Polyomino translation tiling problem.
For $k=1$, this problem has been proven to be decidable(Wijshoff and van Leeuwen \cite{WL}).
Ollinger proved that the problem is undecidable for $k = 11$ by reduction of the domino problem \cite{NO}.
Yang and Zhang improved this result by showing that the problem remains undecidable for $k = 8$ \cite{YZ}, then for $ k=7$ \cite{YZ1}.
In this paper, we show that the problem remains undecidable for $k = 5$.

We may also consider the problem of tiling the plane with $k$ tiles, allowing rotations.
Ollinger proved that the problem is undecidable for $5$ polyominoes \cite{NO}.
Demaine and Langerman showed that the problem is undecidable for $3$ polygons \cite{ES}, in which the polygons are required to rotate in arbitrarily many directions.
This paper shows that the problem is undecidable for $3$ polyominoes.
Since the result only requires the tiles to rotate in $3$ different directions, it slightly improves the result by Demaine and Langerman.

In Section 2, we introduce a new method of adding dents and bumps to polyominoes.
In Section 3, we prove the undecidability of the $5$-Polyomino translation tiling problem. 
In Section 4, we discuss the significance of this result and how it may be expanded.

\section{Labeling of polyominoes} 

Labeling polyominoes is the process of scaling up a set of polyominoes and adding relatively small dents or bumps.
This technique has been used in various papers as a way to enforce certain matching rules, allowing for more complex tilings \cite{NO, YZ, YZ1}.
However, as we show in Section 2.1, the labeling methods used in previous papers have several restrictions.

\subsection{Polyominoes with a matching rule} 

For a polyomino set $P$, let $G$ be an undirected graph with the edges of $P$ as nodes.
$P_G$ is a set of polyominoes equal to ${P}$ where $2$ edges can only match if they are adjacent in $G$.
In other words, in $P_G$, the graph $G$ assigns a matching rule to $P$.
This is a generalization of the matching rules used in previous papers. 

Edges of certain $P_G$’s can be constructed using dents and bumps. Figure $2$ shows $2$ examples.

\begin{figure}[ht]
\centering
\includegraphics[width=1 \textwidth]{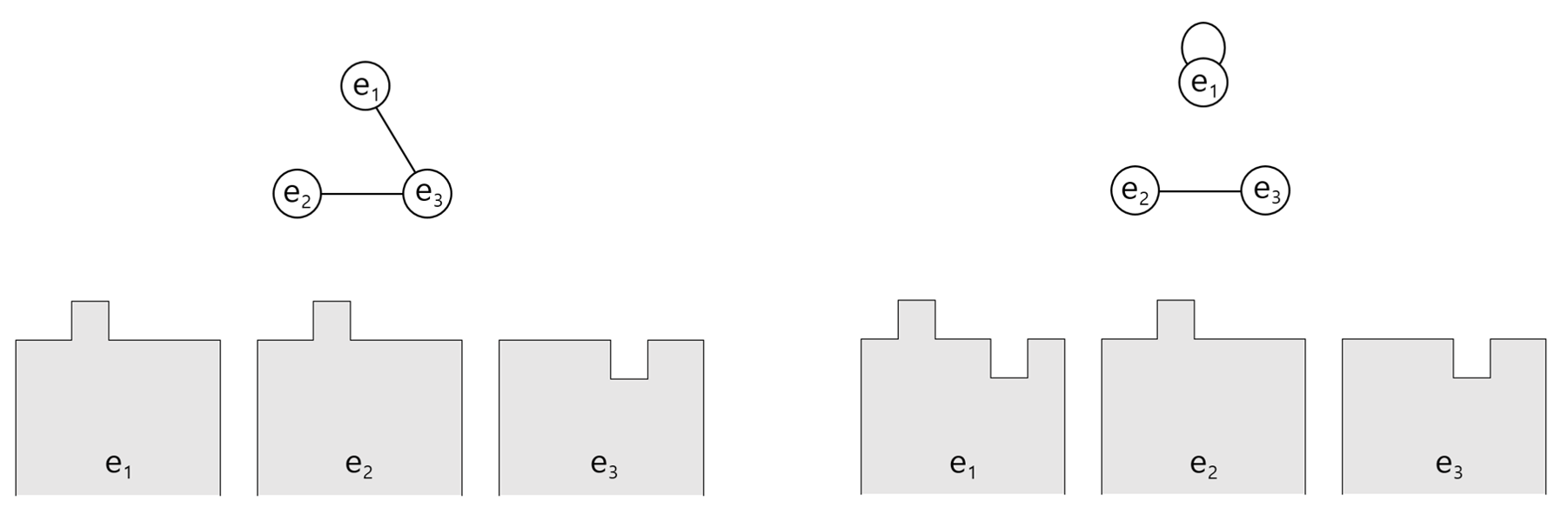}
\captionsetup{width=\textwidth}
\caption{2 sets of labeled edges, each corresponding to a different $G$}
\label{fig:Figure2}
\end{figure}

Using this method, an edge can only match another edge with an identical border.
This prevents the construction of most $P_G$’s.
For example, the graph in Figure $3$ cannot be encoded by dents and bumps as a matching rule:
$e_3$ can match with both $e_1$ and $e_2$, forcing $e_1$ and $e_2$ to be identical.
However, $e_1$ can match with itself while $e_2$ cannot, leading to a contradiction.

\begin{figure}[ht]
\centering
\includegraphics[width=0.18 \textwidth]{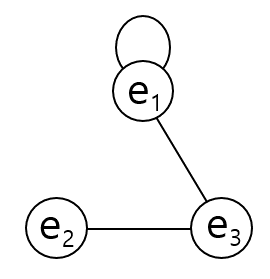}
\caption{ A graph with three nodes $e_1$, $e_2$ and $e_3$}
\label{fig:Figure3}
\end{figure}

In this paper, we use an additional polyomino to fill in gaps when $2$ edges do not match precisely.
We first take a look at the shapes of the labels, then show that these labels can encode ${P}_G$ for an arbitrary $G$.

\subsection{\boldmath Keys and locks } 

A \emph{tooth} is a $9$-omino in the shape of a cross(Figure 4). The shape and size of a \emph{tooth} does not change.

\begin{figure}[ht]
\centering
\includegraphics[width=0.18 \textwidth]{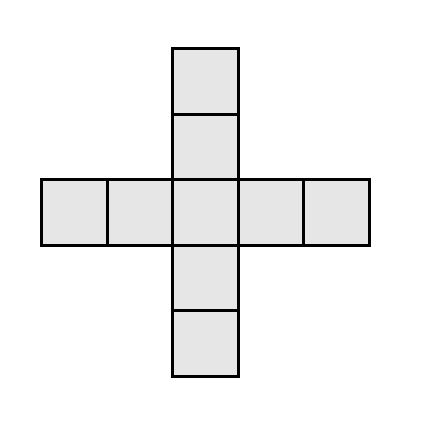}
\caption{A \emph{tooth}}
\label{fig:Figure4}
\end{figure}

A \emph{key} is a bump on an edge(Figure 5). The shape of a \emph{key} is decided by two factors: length and index.
A \emph{key} which has length $l$ and index $k$ is a $(6l-2) \times 1$ rectangle with a single \emph{tooth}
attached on the $(6k-2)$th square(starting from the square closest to the edge).

\begin{figure}[ht]
\centering
\includegraphics[width=0.37 \textwidth]{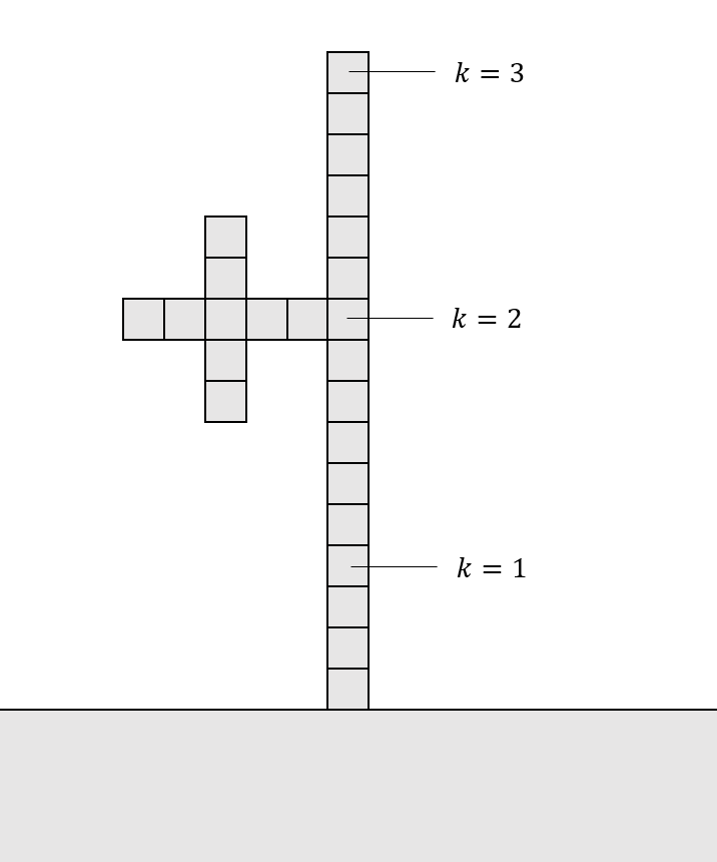}
\caption{A \emph{key} with length 3 and index 2}
\label{fig:Figure5}
\end{figure}

A \emph{lock} is a dent on an edge(Figure 6). The negative space of a \emph{lock} is similar to a \emph{key}, but instead of an index,
a \emph{lock} is decided by an index set.
Negative space of a \emph{lock} which has length $l$ and index set $S$ is a $(6l-2) \times 1$ rectangle with a \emph{tooth}
attached on the $(6x-2)$th square(starting from the square closest to the edge), for every $x \in S$.

\begin{figure}[ht]
\centering
\includegraphics[width=0.3 \textwidth]{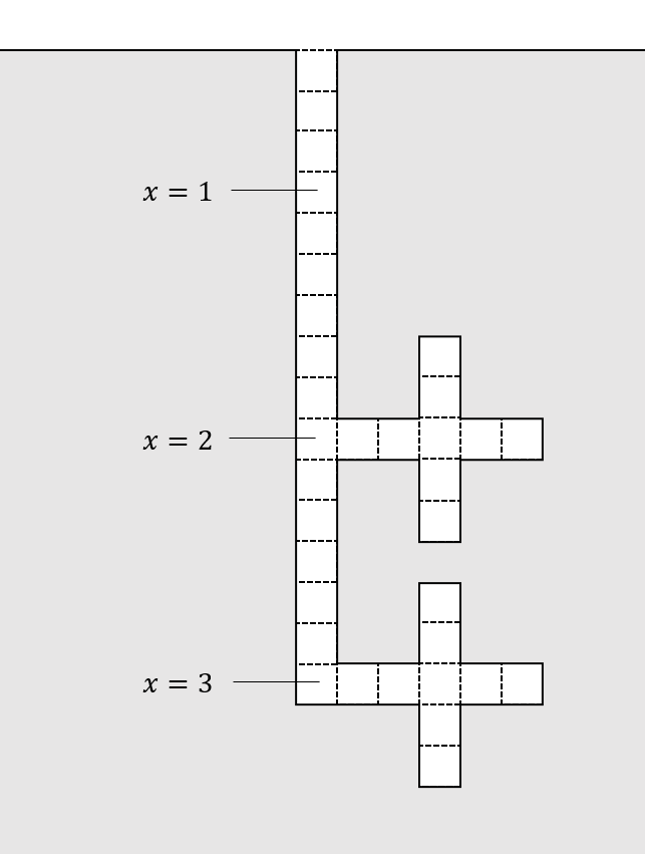}
\caption{A \emph{lock} with length 3 and index set $\{2,3\}$}
\label{fig:Figure6}
\end{figure}

\begin{lemma} 
It is impossible to tile the plane with only teeth.
\end{lemma}
\begin{proof}
The inward corner of a \emph{tooth}(the location of the left dot in Figure $7$) can only be filled by the outermost square of another \emph{tooth}.
This creates a dent of $2$ unit squares(the location of the right two dots in Figure $7$), which cannot be both filled.
\end{proof}

\begin{figure}[ht]
\centering
\includegraphics[width=0.3 \textwidth]{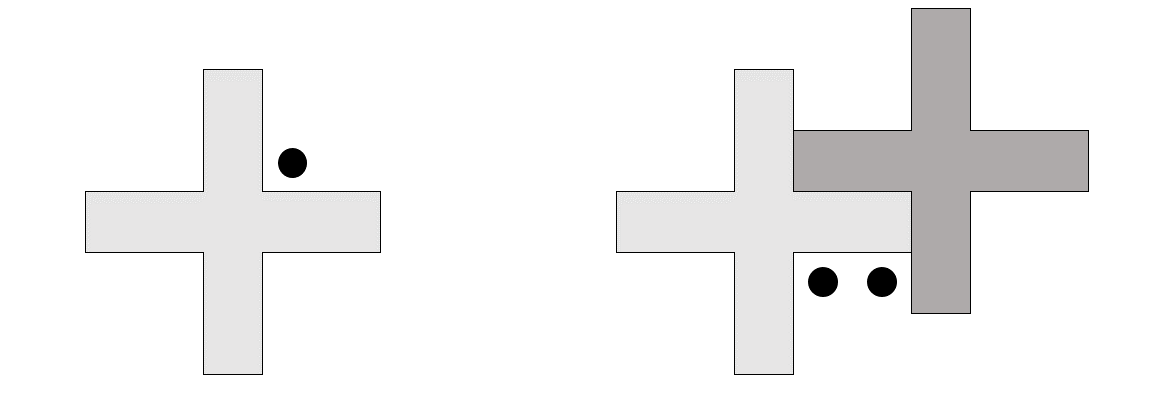}
\caption{Tiling possibilities of a \emph{tooth}}
\label{fig:Figure7}
\end{figure}

\begin{lemma} 
A key and a lock can match if and only if the index of the key belongs in the index set of the lock.
\end{lemma}

\begin{proof}
For a \emph{key} to be able to match a \emph{lock}, the \emph{tooth} attached to the \emph{key} must be positioned in
one of the \emph{tooth}-shaped holes in the \emph{lock}. Therefore, the index set of the \emph{lock} must include the index of the \emph{key}.
If the index set of the \emph{lock} includes the index of the \emph{key}, the \emph{key} can be placed inside the \emph{lock}, possibly leaving holes.
These holes can then be filled by the \emph{tooth}.
\end{proof}

\subsection{KL labeling} 

Given a set of polyominoes $P$, by scaling the polyominoes by a sufficiently large factor,
it is possible to label the center of every edge with one \emph{key} and one \emph{lock} of the same length.
$KL$ labeling of $P$ is a set of polyominoes obtained by applying this form of labeling to $P$,
in addition to a \emph{tooth}(Figure $8$).
We will denote $KL$ labeling of $P$ as $KL(P)$.
The index of the \emph{key} on edge $e$ will be denoted $k_e$,
and the index set of the \emph{lock} on edge $e$ will be denoted $S_e$.

Throughout this paper, if a polyomino $p$ belongs in $KL(P)$, an edge of $p$ refers to the labeled edge as a whole.(We do not consider
edges of single squares in $p$.)

\begin{figure}[ht]
\centering
\includegraphics[width=0.3 \textwidth]{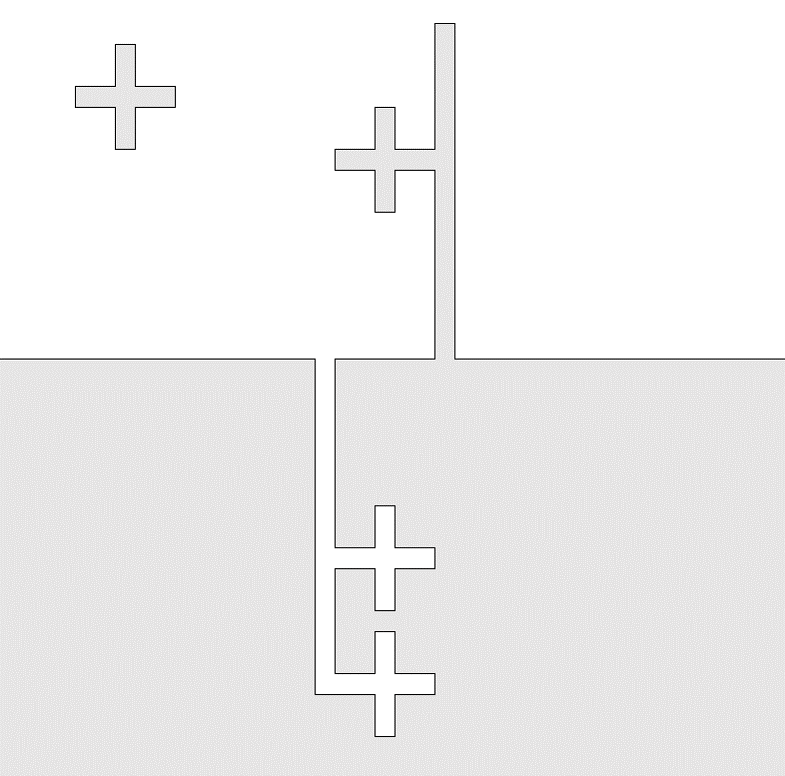}
\captionsetup{width=0.8 \textwidth}
\caption{Example of a $KL$ labeling on a polyomino. The figure only shows one edge from the original polyomino.
The edge is labeled with a \emph{key} of index 2 and a \emph{lock} of index set \{2, 3\}.}
\label{fig:Figure8}
\end{figure}

\begin{lemma} 
For a polyomino set $P$, $KL(P)$ only permits tilings that correspond to a tiling by $P$.
\end{lemma}
\begin{proof}
Let us consider tiling the plane with $KL(P)$. By Lemma 2.1, a labeled version of $P$ must be used at least once.
Let us call one such polyomino $p_0$. Every edge of $p_0$ consists of a \emph{key} and a \emph{lock}.
The innermost square of the \emph{lock}(the location of the dot in Figure $9$)
can only be filled by a \emph{key}. Therefore, every edge of $p_0$ must be adjacent to another labeled version of $P$.
The placement of the \emph{keys} and \emph{locks} forces adjacent polyominoes to be aligned(Figure $10$).
Repeating this process, every labeled version of $P$ in the tiling must lie on a lattice of side length $\alpha$,
where $\alpha$ is the scaling factor in the process of constructing $KL(P)$.
Therefore, any tiling by $KL(P)$ must correspond to a tiling by $P$.
\end{proof}

\begin{figure}[ht]
\centering
\includegraphics[width=0.3 \textwidth]{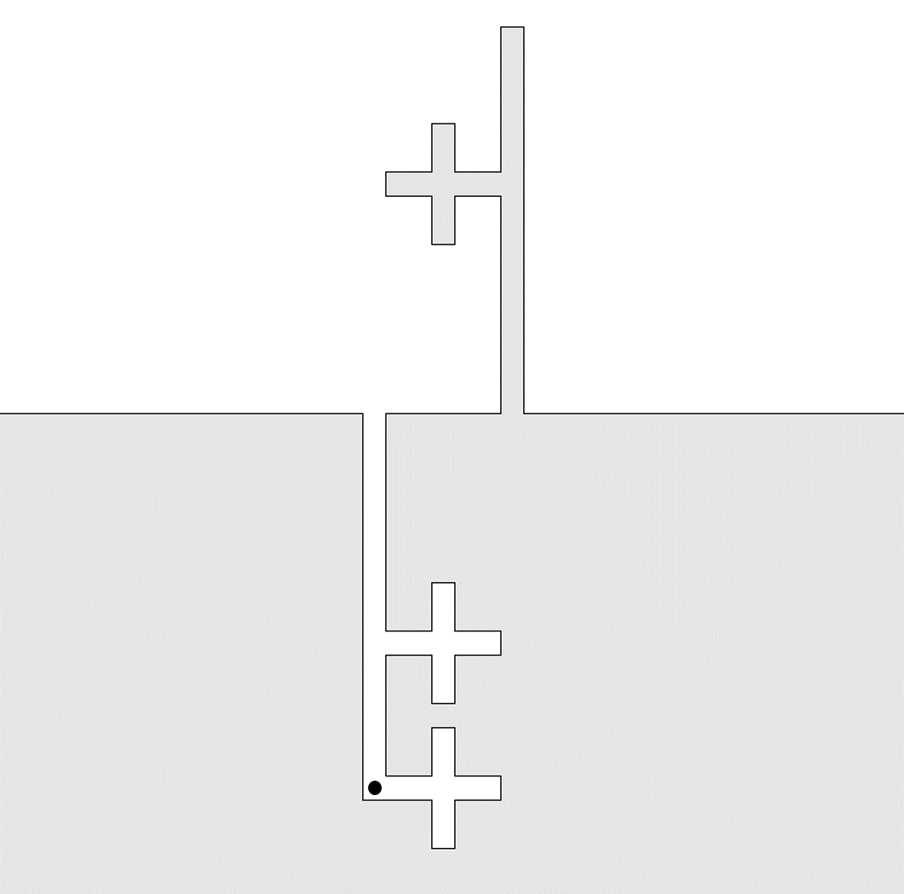}
\caption{Tiling possibilities of a \emph{lock}}
\label{fig:Figure9}
\end{figure}

\begin{figure}[ht]
\centering
\includegraphics[width=0.5 \textwidth]{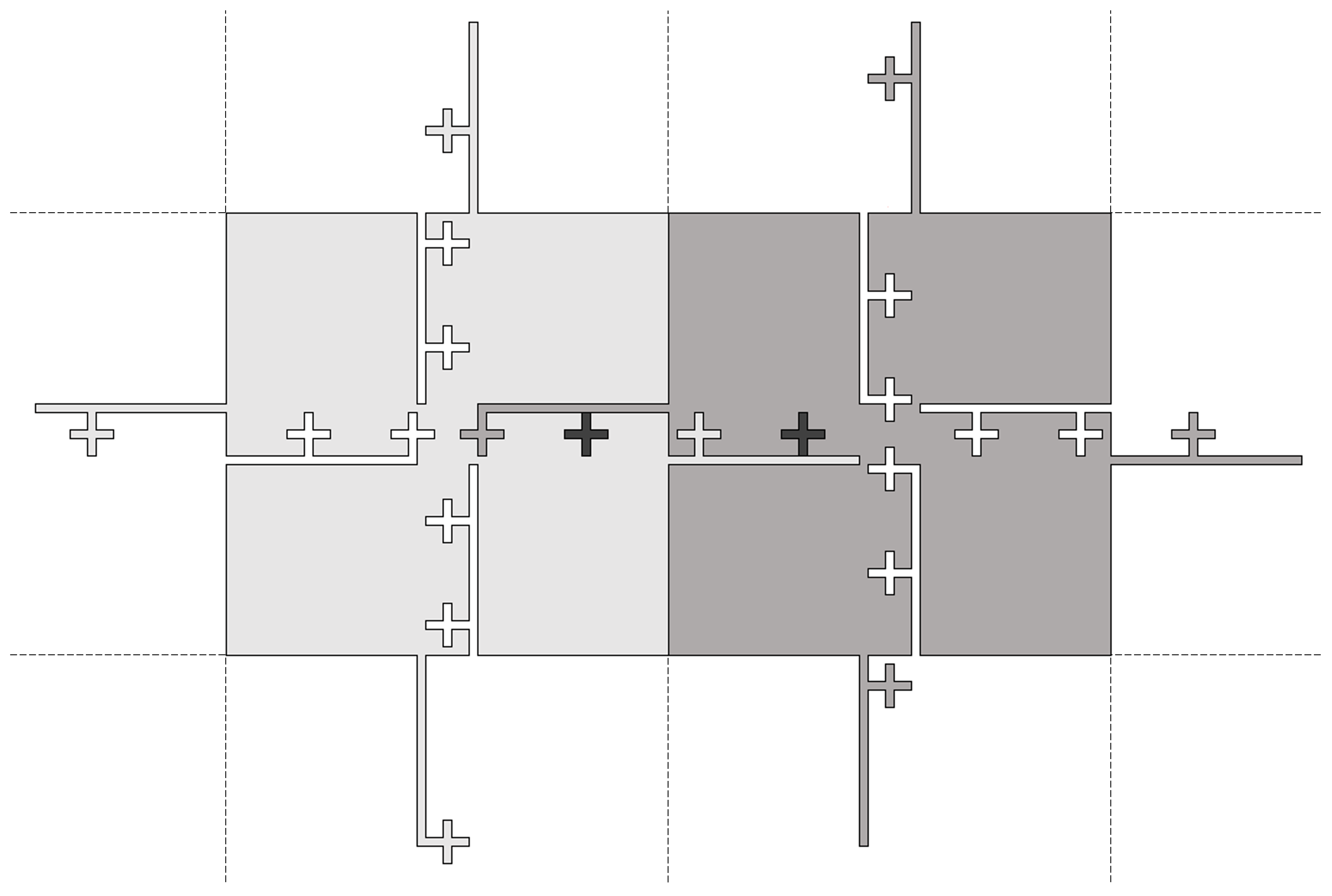}
\captionsetup{width=0.65 \textwidth}
\caption{Adjacent edges in the tiling of $KL(P)$, where $P$ only consists of a monomino}
\label{fig:Figure10}
\end{figure}

\begin{theorem} 
For a given $P_G$, it is always possible to construct $KL(P)$ such that
the set of tilings by $KL(P)$ has a one-to-one correspondence with the set of tilings by $P_G$.
\end{theorem}

\begin{proof}
We construct $KL_G(P)$ such that for every edge $e_i$, \\
$S_{e_i} = \{k_{e_j} \mid e_i \text{ and } e_j \text{ are adjacent in } G\}$.
By Lemma 2.3, every tiling by $KL_G(P)$ must correspond to a tiling by $P$.
By Lemma 2.2, two edges $e_i$ and $e_j$ of $KL_G (P)$ can match if and only if $k_{e_i} \in S_{e_j}$ and $k_{e_j} \in S_{e_i}$;
this is equivalent to $e_i$ and $e_j$ being connected in $G$. 
Therefore, every tiling by $KL_G (P)$ corresponds to a tiling by $P$ such that the resulting tiling satisfies the matching rules of $P_G$.
Likewise, every tiling by $P_G$ can be turned into a tiling by $KL_G (P)$ by applying $KL$ labeling.
\end{proof}

\begin{figure}[ht]
\centering
\includegraphics[width=0.75 \textwidth]{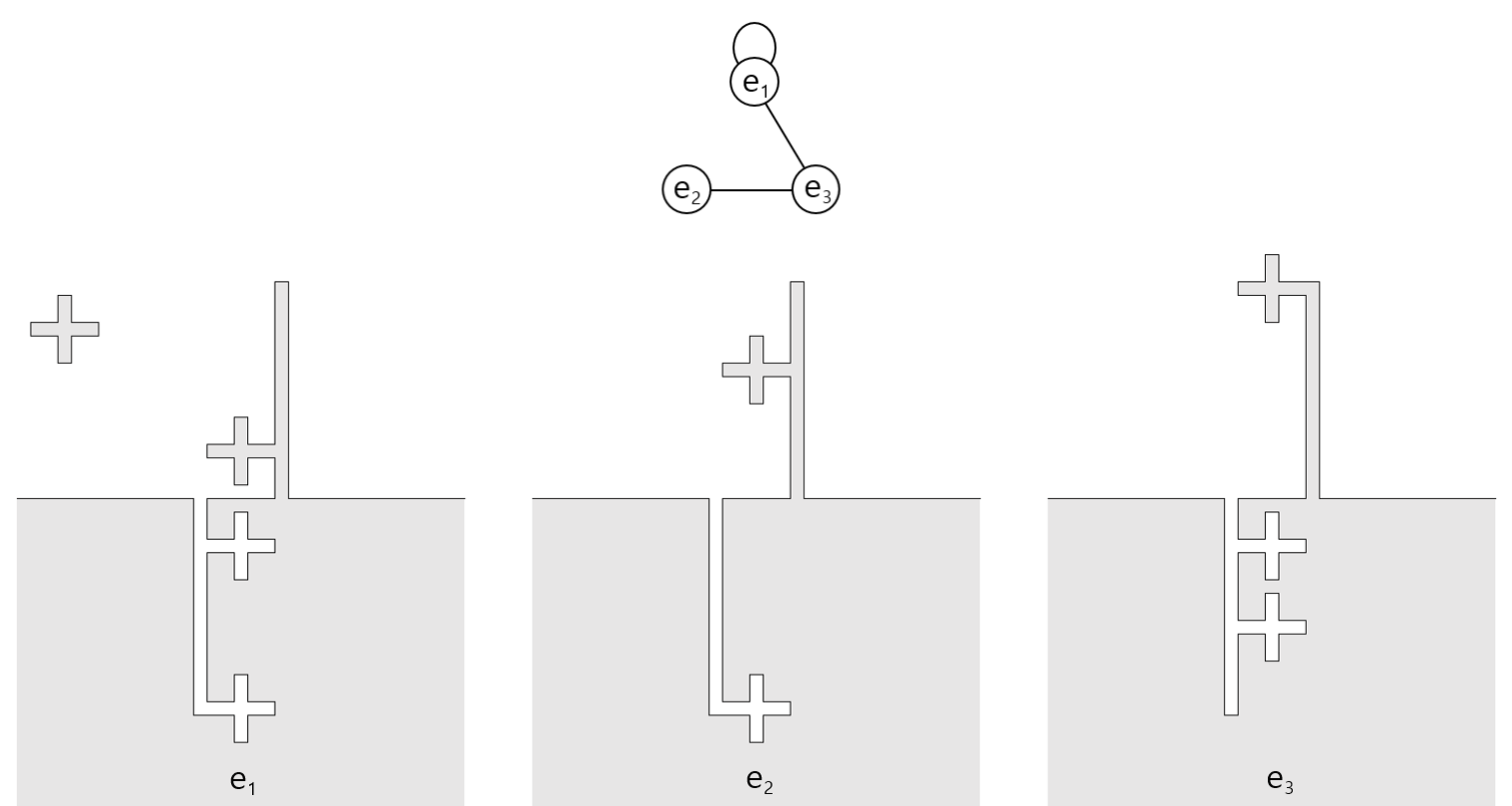}
\captionsetup{width=0.65 \textwidth}
\caption{A graph with three nodes $e_1$, $e_2$ and $e_3$, and the corresponding $KL$ labeling}
\label{fig:Figure11}
\end{figure}

\section{Encoding of Wang tiles} 

The conditions for the decidability problem of tiling the plane with polyominoes can be set in two ways.
The $k$-polyomino translation tiling problem asks if it is decidable whether a set of  $k$ polyominoes can tile the plane by translation.
The $k$-polyomino tiling problem asks the same question while the rotation of polyominoes is allowed.
Ollinger proved the undecidability of the $5$-Polyomino tiling problem,
which also gave a proof for the undecidability of the $11$-Polyomino translation tiling problem \cite{NO}.
Similarly, we will give a proof that the $3$-Polyomino tiling problem is undecidable.
This result can then be used to prove the undecidability of the $5$-Polyomino translation tiling problem.

\begin{theorem} 
The $3$-Polyomino tiling problem is undecidable.
\end{theorem}
We will prove Theorem 3.1 by reduction of the {\it domino problem}. 
In Ollinger’s proof of the undecidabililty of the $5$-Polyomino tiling problem,
Ollinger encoded an arbitrary set of Wang tiles using \emph{meat}, \emph{jaw}, \emph{filler}, \emph{wire} and \emph{tooth} \cite{NO}.
Using KL labeling, it is possible to reduce the number of polyominoes while maintaining a similar structure.
We will show that a set of $n$ Wang tiles with $m$ distinct colors can be encoded using \emph{tooth}, \emph{rod} and \emph{blade}.

\begin{itemize}
\setlength{\leftskip}{-2em}
\item[] \textbf{tooth}: a part of the $KL$ labeling process
\item[] \textbf{rod}: $KL$ labeled version of a $((6n-2)m+6) \times 2$ rectangle; replaces the $meat$, $filler$, and $wire$ in Ollinger’s construction
\item[] \textbf{blade}: $KL$ labeled version of two $((6n-2)m+6) \times ((4n-4)m+6)$ rectangles attached to a $((6n-2)m+6) \times 2$ rectangle;
replaces the \emph{jaw} in Ollinger’s construction
\end{itemize}
Figure $12$ shows the \emph{blade} and the \emph{rod} which encode the Wang tiles in Figure 1.

\begin{figure}[ht]
\centering
\includegraphics[width=0.75 \textwidth]{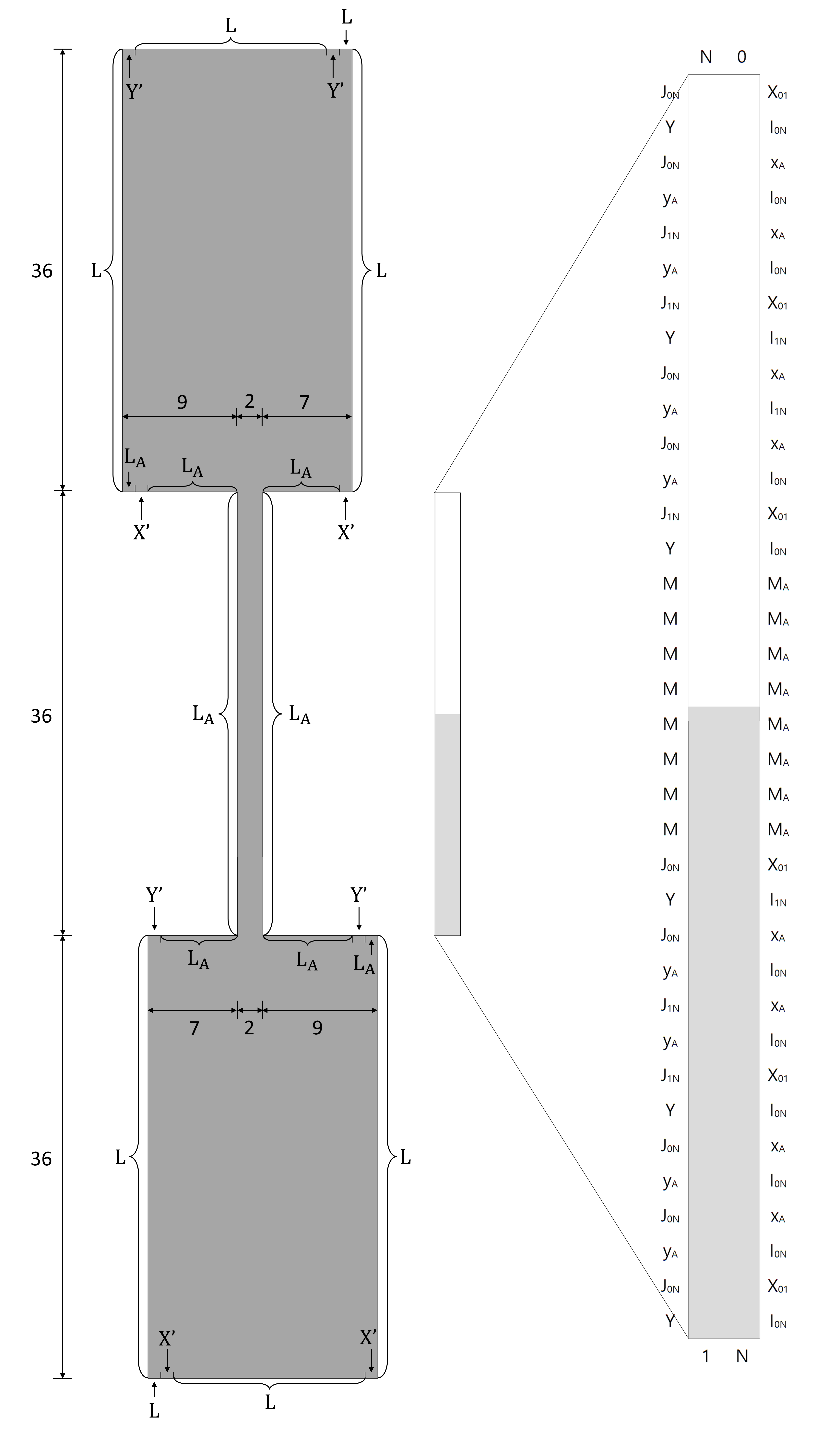}
\caption{The \emph{blade} and the \emph{rod}}
\label{fig:Figure12}
\end{figure}

\subsection{Notation of labels} 

A total of $17$ different labels are used in constructing the \emph{blade} and the \emph{rod}.
The length of all \emph{keys} and \emph{locks} is $17$, being $100$ units long.
The minimum required scaling factor in this case(where the \emph{locks} do not overlap) is 207.
We construct the labels using this minimum scaling factor.

Each label is given a distinct index $1 \textendash 17$ for the \emph{key}; the index sets given for the locks were constructed using the method in Theorem 2.4.
Although it is possible to provide the matching rule $G$ that was applied,
we will instead use the index set for each label(shown in Table 3) since their representations are almost identical.

Table $1$ and Table $2$ give an overview of the labels and their roles.

\begin{table}[ht] 
\renewcommand{\arraystretch}{1.3}
\centering
\resizebox{\textwidth}{!}{%
\begin{tabular}{|c|l|}
\hline
label &  \multicolumn{1}{c|} {description} \\ \hline
\textbf{$0$} & \multirow{2}{*}{\begin{tabular}[c]{@{}l@{}}located at either end of a \emph{rod}; transfers 1 bit of information between \\
simulated Wang tiles and enforces their matching rule\end{tabular}} \\ \cline{1-1}
\textbf{$1$} &  \\ \hline
\textbf{$N$} & \begin{tabular}[c]{@{}l@{}}located at either end of a \emph{rod}; acts as a ‘neutral bit’ to erase information \\
when information is not necessary\end{tabular} \\ \hline
\end{tabular}%
}
\caption{Properties of labels $0$, $1$ and $N$}
\end{table}

\begin{table}[ht] 
\renewcommand{\arraystretch}{1.5}
\centering
\resizebox{\textwidth}{!}{%
\begin{tabular}{|c|l|l|}
\hline
label & \multicolumn{1}{c|} {$S_{label}$} & \multicolumn{1}{c|}{description} \\ \hline
\textbf{$M$} & $\{ k_e \mid e \in \{M, L\} \}$ &
  \begin{tabular}[c]{@{}l@{}}located in the middle of the long sides of a \emph{rod}; \\ acts as a spacer\end{tabular} \\ \hline

\textbf{$X$} & $\{ k_e \mid e \in \{I, J, L, X'\} \}$ &
\multirow{2}{*}{\begin{tabular}[c]{@{}l@{}}located at the long sides of a \emph{rod}; acts as a marker\\
to ensure that only 1 section of information can be selected\end{tabular}} \\ \cline{1-2}
\textbf{$Y$} & $\{ k_e \mid e \in \{I, L, Y'\} \}$ & \\ \hline

\textbf{$x$} & $\{ k_e \mid e \in \{I, J, L\} \}$ &
\multirow{2}{*}{\begin{tabular}[c]{@{}l@{}}located at the long sides of a \emph{rod}; \\
matches with the unused $N$ when information is being read\end{tabular}} \\ \cline{1-2}
\textbf{$y$} & $\{ k_e \mid e \in \{I, L\} \}$ & \\ \hline

\textbf{$I$} & $\{ k_e \mid e \in \{X, Y, x, y, L\} \}$ &
\multirow{2}{*}{\begin{tabular}[c]{@{}l@{}}located at the long sides of a \emph{rod}, separated by $X, Y, x$ or $y$; \\
contains information of Wang tiles to be simulated\end{tabular}} \\ \cline{1-2}
\textbf{$J$} & $\{ k_e \mid e \in \{X, x, L\} \}$ & \\ \hline

\textbf{$L$} & $\{ k_e \mid e \in \{M, X, Y, x, y, I, J\} \}$ &
  \begin{tabular}[c]{@{}l@{}}located around the \emph{blade}; \\can only match with the long sides of a \emph{rod}\end{tabular} \\ \hline

\textbf{$X’$} & $\{ k_e \mid e \in \{X\} \}$ &
\multirow{2}{*}{\begin{tabular}[c]{@{}l@{}}located at the ends of the \emph{blade}; forces the \emph{blade} to match\\
correctly when selecting a section of information\end{tabular}} \\ \cline{1-2}
\textbf{$Y’$} & $\{ k_e \mid e \in \{Y \} \}$ & \\ \hline
\end{tabular}%
}
\captionsetup{width=\textwidth}
\caption{Properties of labels $M$, $X$, $Y$, $x$, $y$, $I$, $J$, $L$, $X’$ and $Y’ $}
\end{table}

Note that none of the labels in Table 2 can match with labels $0$, $1$, or $N$ by default.
Instead, variants of labels in Table 2 are used as the following:\\
If $Z$ is a label, $Z_{0N}$, $Z_{1N}$, $Z_{01}$, and $Z_A$ are labels where
${S_{Z_{0N}} = S_Z \cup \{k_0, k_N\}}$,
${S_{Z_{1N}} = S_Z \cup \{k_1, k_N\}}$, ${S_{Z_{01}} = S_Z \cup \{k_0, k_1\}}$,
and ${S_{Z_A} = S_Z \cup \{k_0, k_1, k_N\}}$.

Table 3 shows a complete list of $17$ labels, including the variants used;
Table 4 shows their boundary word when the \emph{key} is oriented upwards(going from left to right, visiting the \emph{lock} first).

\begin{table}[ht] 
{ \footnotesize
\renewcommand{\arraystretch}{1.27}
\centering
\begin{tabular}{|c|l|}
\hline
label &  \multicolumn{1}{c|} {$S_{label}$} \\ \hline
\textbf{$0$}   & \{ $k_e \mid e \in \{M_A, X_{01}, x_A , y_A, I_{0N} , J_{0N}, L_A \} \}$ \\ \hline
\textbf{$1$}   & \{ $k_e \mid e \in \{M_A, X_{01}, x_A , y_A, I_{1N} , J_{1N}, L_A\} \}$ \\ \hline
\textbf{$N$}   & \{ $k_e \mid e \in \{M_A, x_A , y_A, I_{0N}, I_{1N}, J_{0N}, J_{1N}, L_A \} \}$ \\ \hline
\textbf{$M$}   & \{ $k_e \mid e \in \{M, M_A, L, L_A \} \}$ \\ \hline
\textbf{$M_A$}  & \{ $k_e \mid e \in \{M, M_A, L, L_A, 0, 1, N \} \}$ \\ \hline
\textbf{$X_{01}$} & \{ $k_e \mid e \in \{I_{0N}, I_{1N}, J_{0N}, J_{1N}, L, L_A, X', 0, 1 \}  \}$ \\ \hline
\textbf{$Y$}   & \{ $k_e \mid e \in \{I_{0N}, I_{1N}, L, L_A, Y' \} \}$ \\ \hline
\textbf{$x_A$}  & \{ $k_e \mid e \in \{I_{0N}, I_{1N}, J_{0N}, J_{1N},L, L_A, 0, 1, N \} \}$ \\ \hline
\textbf{$y_A$}  & \{ $k_e \mid e \in \{I_{0N}, I_{1N}, L, L_A, 0, 1, N \} \}$ \\ \hline
\textbf{$I_{0N}$} & \{ $k_e \mid e \in \{X_{01}, Y, x_A, y_A,L, L_A, 0, N \} \}$  \\ \hline
\textbf{$I_{1N}$} & \{ $k_e \mid e \in \{X_{01}, Y, x_A, y_A, L, L_A, 1, N \} \}$  \\ \hline
\textbf{$J_{0N}$} & \{ $k_e \mid e \in \{X_{01}, x_A , L, L_A, 0, N \} \}$ \\ \hline
\textbf{$J_{1N}$} & \{ $k_e \mid e \in \{X_{01}, x_A , L, L_A, 1, N \} \}$ \\ \hline
\textbf{$L$}   & \{ $k_e \mid e \in \{M, M_A, X_{01}, Y, x_A , y_A, I_{0N}, I_{1N}, J_{0N}, J_{1N} \} \}$ \\ \hline
\textbf{$L_A$}  & \{ $k_e \mid e \in \{M, M_A, X_{01}, Y, x_A , y_A, I_{0N}, I_{1N}, J_{0N}, J_{1N}, 0, 1, N \} \}$ \\ \hline
\textbf{$X’$}  & \{ $k_e \mid e \in \{X_{01} \} \}$ \\ \hline
\textbf{$Y’$}  & \{ $k_e \mid e \in \{Y \} \}$ \\ \hline
\end{tabular}%
}
\caption{List of all labels}
\end{table}

\begin{table}[ht] 
{\small
\renewcommand{\arraystretch}{1.25}
\centering

\begin{tabular}{|l|}
\hline
{$t=r^2 d^2 ru^2 r^2 ul^2 u^2 ld^2 l^2$, \ ${\bar{t}} =l^2 d^2 lu^2 l^2 ur^2 u^2 rd^2 r^2$} \\ \hline
{$r_0 =r^{100} d^{100} r(u^{12} tu^{17} tu^{11} tu^{5} tu^{5} tu^{11} tu^{5} tu^{27} )r^{5} (u^{3} {\bar{t}} u^{96} )rd^{100} r^{100}$} \\ \hline
{$r _{1} =r^{100} d ^{100} r(u ^{12} tu ^{11} tu ^{11} tu ^{11} tu ^{5} tu ^{11} tu ^{5} tu ^{27} )r ^{5} (u ^{9} {\bar{t}} u ^{90} )rd ^{100} r ^{100}$} \\ \hline
{$r _{N} =r^{100} d ^{100} r(u ^{12} tu ^{11} tu ^{5} tu ^{5} tu ^{5} tu ^{5} tu ^{5} tu ^{17} tu ^{27} )r ^{5} (u ^{15} {\bar{t}} u ^{84} )rd ^{100} r ^{100}$} \\ \hline
{$r _{M} =r^{100} d ^{100} r(u ^{12} tu ^{5} tu ^{53} tu ^{5} tu ^{21} )r ^{5} (u ^{21} {\bar{t}} u ^{78} )rd ^{100} r ^{100}$} \\ \hline
{$r _{M _{A}} =r^{100} d ^{100} r(u ^{12} tu ^{5} tu ^{53} tu ^{5} tu ^{5} tu ^{5} tu ^{5} tu ^{3} )r ^{5} (u ^{27} {\bar{t}} u ^{72} )rd ^{100} r ^{100}$} \\ \hline
{$r _{X _{01}} =r^{100} d^{100} r(u^{6} tu^{5} tu^{5} tu^{5} tu^{5} tu^{5} tu^{5} tu^{47} tu^{5} tu^{3} )r^{5} (u^{33} {\bar{t}} u^{66} )rd^{100} r^{100}$} \\ \hline
{$r _{Y} =r^{100} d ^{100} r(tu ^{11} tu ^{5} tu ^{17} tu ^{5} tu ^{51} )r ^{5} (u ^{39} {\bar{t}} u ^{60} )rd ^{100} r ^{100}$} \\ \hline
{$r _{x _{A}} =r^{100} d^{100} r(u^{12} tu^{5} tu^{5} tu^{5} tu^{5} tu^{5} tu^{41} tu^{5} tu^{5} tu^{3} )r^{5} (u^{45} {\bar{t}} u^{54} )rd^{100} r^{100}$} \\ \hline
{$r _{y _{A}} =r^{100} d ^{100} r(u ^{12} tu ^{5} tu ^{17} tu ^{5} tu ^{41} tu ^{5} tu ^{5} tu ^{3} )r ^{5} (u ^{51} {\bar{t}} u ^{48} )rd ^{100} r ^{100}$} \\ \hline
{$r _{I _{0N}} =r^{100} d^{100} r(u^{12} tu^{5} tu^{29} tu^{5} tu^{5} tu^{5} tu^{17} tu^{11} tu^{3} )r^{5} (u^{57} {\bar{t}} u^{42} )rd^{100} r^{100}$} \\ \hline
{$r _{I _{1N}} =r^{100} d^{100} r(u^{12} tu^{5} tu^{29} tu^{5} tu^{5} tu^{5} tu^{17} tu^{5} tu^{9} )r^{5} (u^{63} {\bar{t}} u^{36} )rd^{100} r^{100}$} \\ \hline
{$r _{J _{0N}} =r^{100} d ^{100} r(u ^{12} tu ^{5} tu ^{35} tu ^{11} tu ^{17} tu ^{11} tu ^{3} )r ^{5} (u ^{69} {\bar{t}} u ^{30} )rd ^{100} r ^{100}$} \\ \hline
{$r _{J _{1N}} =r^{100} d ^{100} r(u ^{12} tu ^{5} tu ^{35} tu ^{11} tu ^{17} tu ^{5} tu ^{9} )r ^{5} (u ^{75} {\bar{t}} u ^{24} )rd ^{100} r ^{100}$} \\ \hline
{$r _{L} =r^{100} d^{100} r(u^{24} tu^{5} tu^{5} tu^{5} tu^{5} tu^{5} tu^{5} tu^{5} tu^{5} tu^{5} tu^{21} )r^{5} (u^{81} {\bar{t}} u^{18} )rd^{100} r^{100}$} \\ \hline
{$r _{L _{A}} =r ^{100} d ^{100} r(u ^{24} tu ^{5} tu ^{5} tu ^{5} tu ^{5} tu ^{5} tu ^{5} tu ^{5} tu ^{5} tu ^{5} tu ^{5} tu ^{5} tu ^{5} tu ^{3} )r ^{5} (u ^{87} {\bar{t}} u ^{12} )rd ^{100} r ^{100}$} \\ \hline
{$r _{X'} =r ^{100} d ^{100} r(u ^{66} tu ^{33} )r ^{5} (u ^{93} {\bar{t}} u ^{6} )rd ^{100} r ^{100}$} \\ \hline
{$r _{Y'} =r ^{100} d ^{100} r(u ^{60} tu ^{39} )r ^{5} (u ^{99} {\bar{t}} )rd ^{100} r ^{100}$} \\ \hline
\end{tabular}%
}
\caption{Boundary words of labels}
\end{table}

Using the boundary words in Table $4$ and their rotations,
the boundary words for the polyominoes encoding $n$ Wang tiles with $m$ distinct colors are given in Table 5. 

\begin{table}[ht] 
{
\renewcommand{\arraystretch}{1.7}
\resizebox{\textwidth}{!}{%
\centering
\begin{tabular}{|c|l|}
\hline
\textbf{$tooth$} & $ r ^{2} d ^{2} ru ^{2} r ^{2} ul ^{2} u ^{2} ld ^{2} l ^{2} d $ \\ \hline
\textbf{$rod$} & \begin{tabular}[c]{@{}l@{}} $ \left( r_{X_{01}} \left( r_{I} r_{x_{A}} \right)^{m-1} r_{I_{0N}} \right)^{n} r_{X_{01}} r_{I_{0N}} r_{M_{A}}^{(2n-2)m+2} \left( r_{X_{01}} \left( r_{I} r_{x_{A}} \right)^{m-1} r_{I_{0N}} \right)^{n} r_{X_{01}} r_{I_{0N}} d_{N} d _{1}$ \\
$\left( l_{Y} \left( l_{J} l_{y_{A}} \right) ^{m-1} l_{J_{0N}} \right) ^{n} l_{Y} l_{J_{0N}} l_{M}^{(2n-2)m+2} \left( l_{Y} \left( l_{J} l_{y_{A}} \right) ^{m-1} l_{J_{0N}} \right) ^{n} l_Y l_{J_{0N}} u_{N} u_{0} $ \end{tabular} \\ \hline
\textbf{$blade$} & \begin{tabular}[c]{@{}l@{}} $r_{L}^{(6n-2)m+6} d_{X'} d_{L_{A}}^{(2n-2)m} r_{L_{A}}^{(6n-2)m+6} u_{L_{A}}^{(2n-2)m+1} u_{Y'} u_{L_{A}} r_{L}^{(6n-2)m+6} d_{X'} d_{L}^{(4n-4)m+3} d _{X'} d_{L}$ \\
$l_{L}^{(6n-2)m+6} u_{Y'} u_{L_{A}}^{(2n-2)m} l_{L_{A}}^{(6n-2)m+6} d_{L_{A}}^{(2n-2)m+1} d_{X'} d_{L_{A}} l_{L}^{(6n-2)m+6} u_{Y'} u_{L}^{(4n-4)m+3} u_{Y'} u_{L} $ \end{tabular} \\ \hline
\end{tabular}%
}}
\captionsetup{width=\textwidth}
\caption{Boundary words of \emph{tooth}, \emph{blade}, and \emph{rod}($r_I$’s can be either $r_{I_{0N}}$ or $r_{I_{1N}}$
and $l_J$’s can be either $l_{J_0N}$ or $l_{J_N}$ depending on the coloring of the Wang tiles)}
\end{table}

\subsection{The tiling structure} 

We begin with an introduction to the general structure of the tiling. When \emph{blades} are arranged in a pattern as in Figure $13$,
large rectangular spaces are created. Each of these rectangles will encode a single Wang tile.

\begin{figure}[ht]
\centering
\includegraphics[width=0.65 \textwidth]{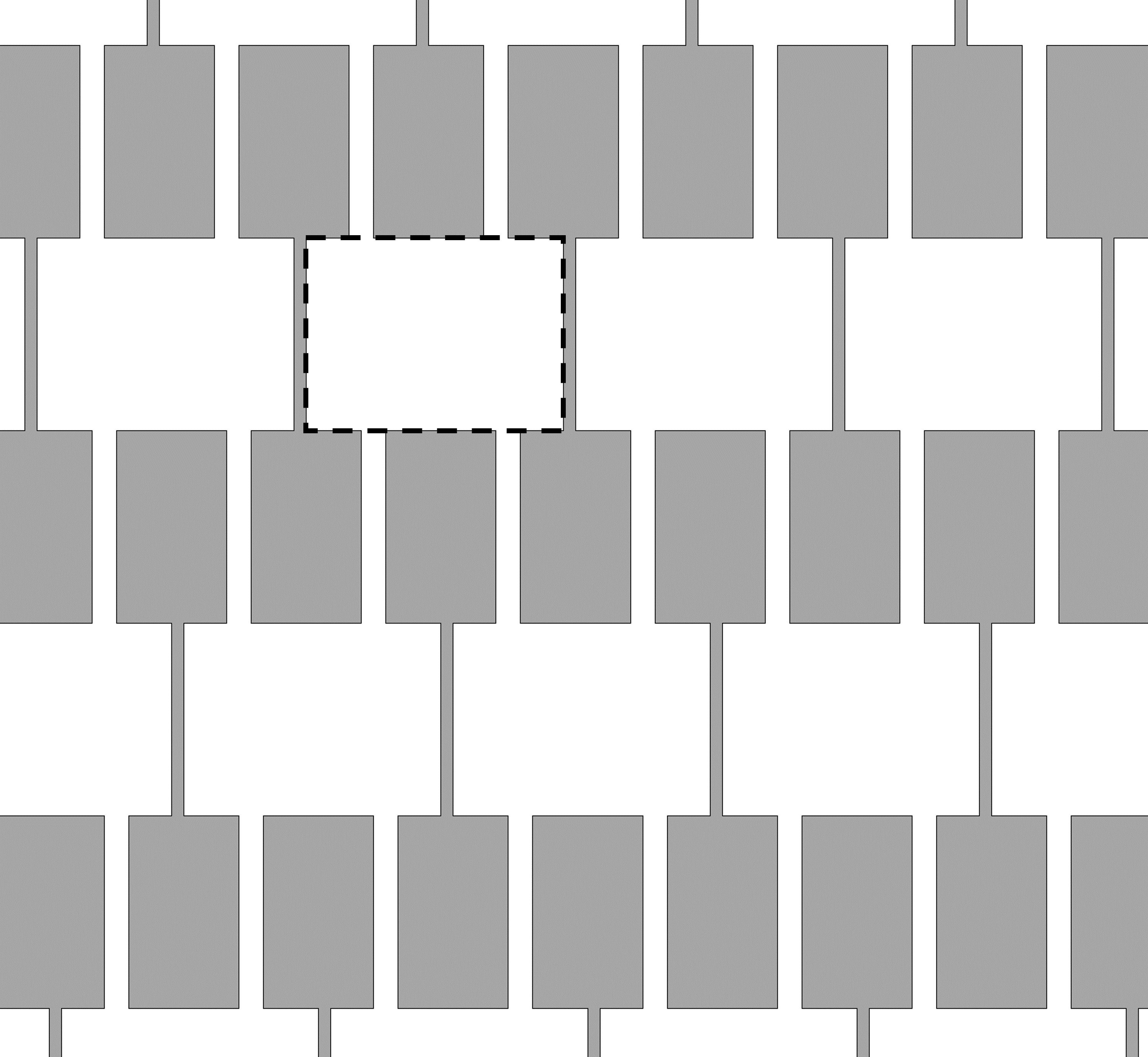}
\caption{A pattern of \emph{blades} forming rectangular spaces}
\label{fig:Figure13}
\end{figure}

In order to distinguish \emph{rods} in different positions, we will denote each \emph{rod} with its corresponding polyomino from Ollinger’s construction.
We classify rods into 3 types: \emph{meat rod}, \emph{filler rod}, and \emph{wire rod}.

The rectangular spaces are filled with \emph{meat rods} and \emph{filler rods}.
The arrangement of these \emph{rods} determines which Wang tile is selected.
Figure $14$ shows 2 possible arrangements of the \emph{rods}.

\begin{figure}[ht]
\centering
\includegraphics[width=0.9 \textwidth]{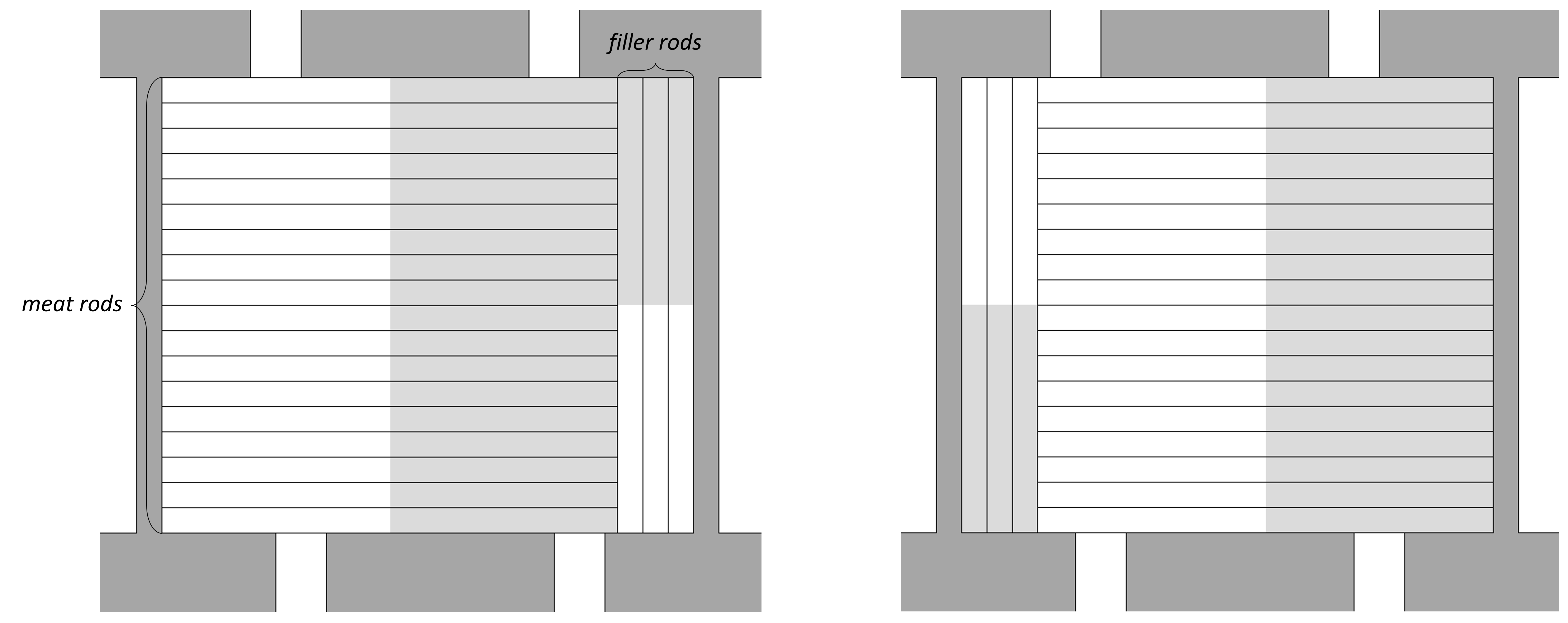}
\captionsetup{width=\textwidth}
\caption{$2$ different arrangements of \emph{rods} filling up the rectangular space}
\label{fig:Figure14}
\end{figure}

Note that in the example used, $n = 2$ so there are only $2$ valid arrangements of the \emph{rods}.
In a case where $n>2$, the \emph{meat rods} can be located between \emph{filler rods} as well.
Figure $15$ shows an example where $n = 4$ and $m = 3$.

\begin{figure}[ht]
\centering
\includegraphics[width=0.8 \textwidth]{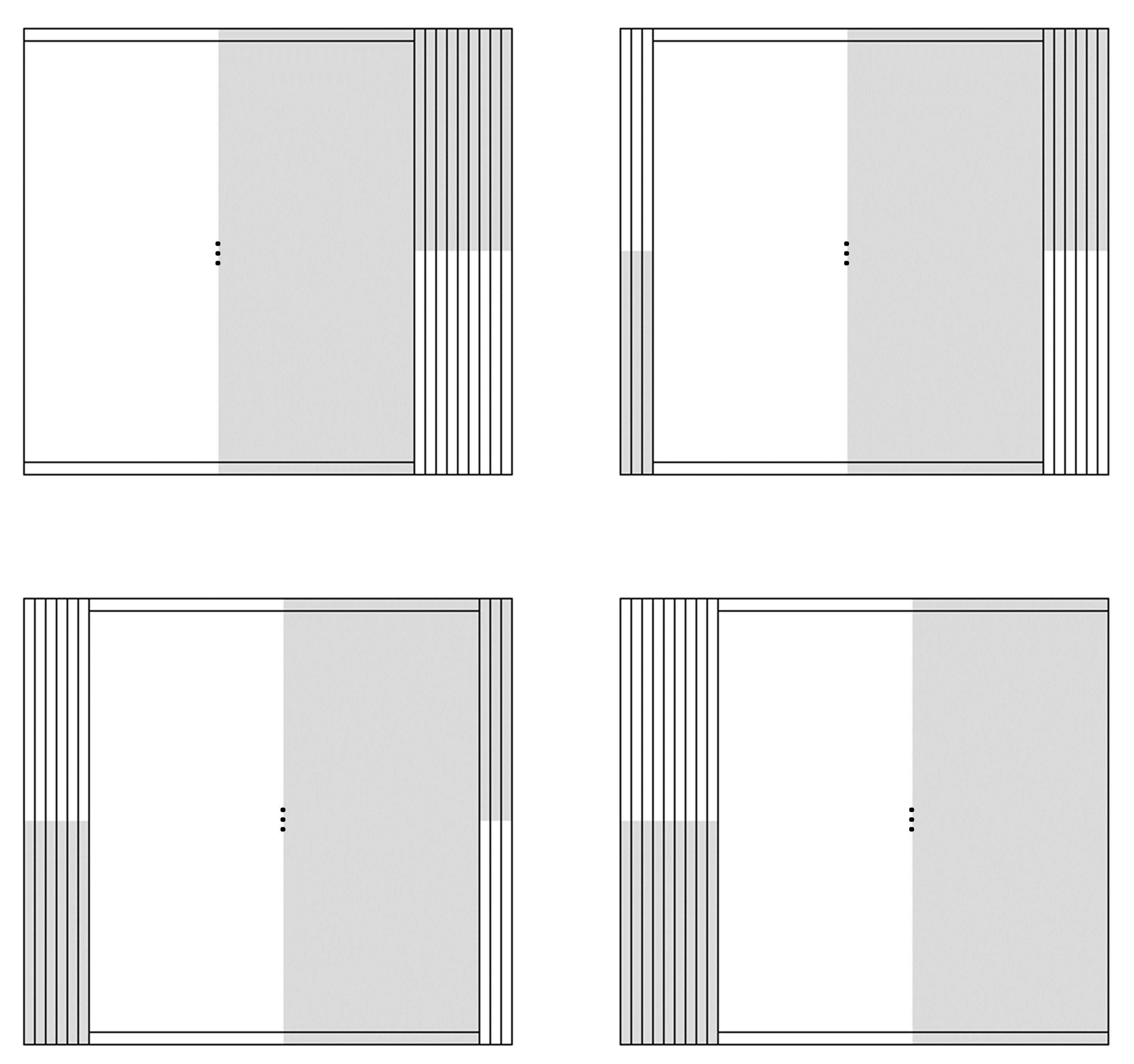}
\caption{4 different arrangements of \emph{rods}}
\label{fig:Figure15}
\end{figure}

In both of the cases in Figure $14$, $4$ different sections of the \emph{meat rods} are exposed.
These correspond to the $4$ sides of the selected Wang tile.
$m-1$ \emph{wire rods} are placed in each of these spaces to transfer information to the adjacent Wang tile(Figure $16$).
This ensures that two Wang tiles can be adjacent only if their colors match along the edge.

\begin{figure}[ht]
\centering
\includegraphics[width=0.22 \textwidth]{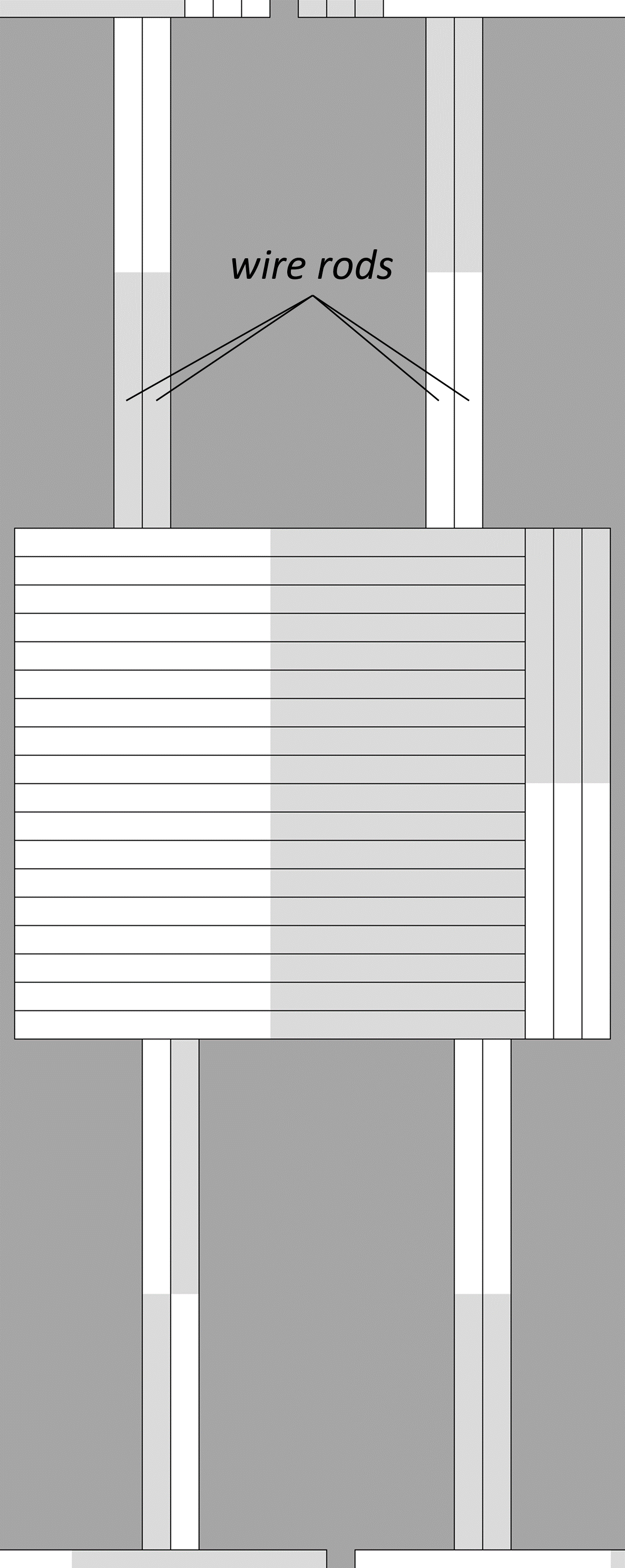}
\caption{\emph{Wire rods} connecting adjacent Wang tiles}
\label{fig:Figure16}
\end{figure}
Figure $17$ shows the complete tiling pattern.

\begin{figure}[ht]
\centering
\includegraphics[width=0.4 \textwidth]{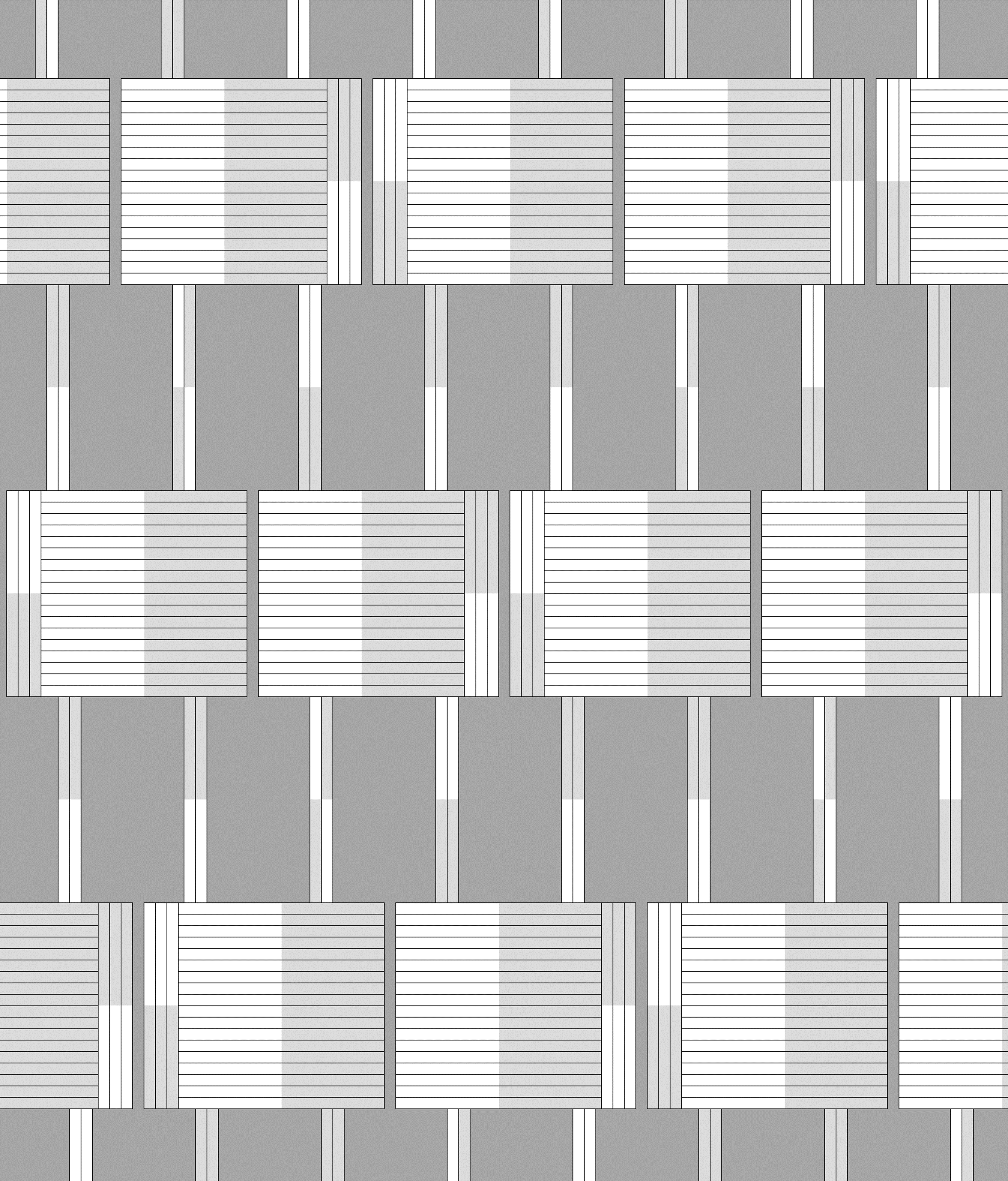}
\captionsetup{width=0.75 \textwidth}
\caption{Tiling of the plane with \emph{teeth}, \emph{rods} and \emph{blades}; \emph{teeth} cannot be seen at this scale.}
\label{fig:Figure17}
\end{figure}

\subsection{Storing and reading information} 

The information of the Wang tiles is stored in labels $I_{0N}$, $I_{1N}$, $J_{0N}$ and $J_{1N}$. 
However, not all of these labels contain information:
in Figure $18$, only the labels marked in black can alter the orientations of \emph{wire rods}, affecting the tiling.
We will refer to these labels as ‘information labels’.

Information labels are grouped into sections, where different sections are separated by labels $X_{01}$ and $Y$.
Each section contains $m-1$ information labels, which encode a single color.
The $i$th color$(1\leq i \leq m)$ from the Wang tiles is encoded as:

$\qquad$ $I_{1N}$ $\cdots$ $I_{1N}$ ($i-1$ copies) $\quad$ $I_{0N}$ $\cdots$ $I_{0N}$ ($m-i$ copies) or

$\quad \; \; \;$ $J_{0N}$ $\cdots$ $J_{0N}$ ($i-1$ copies) $\; \; \,$ $J_{1N}$ $\cdots$ $J_{1N}$ ($m-i$ copies)

$\quad \; \; \;$ (starting from the information label closest to label $0$)

Note that unlike in previous studies where every bit could be set independently, we only allow consequent $0$’s followed by $1$’s, or vice versa.
This is to ensure that the \emph{wire rods} do not form an invalid match.

Sections adjacent to each other represent edges of the same position from different Wang tiles.
Figure $18$ shows how each edge of the Wang tiles in Figure $1$ is encoded.

\begin{figure}[ht]
\centering
\includegraphics[width=0.9 \textwidth]{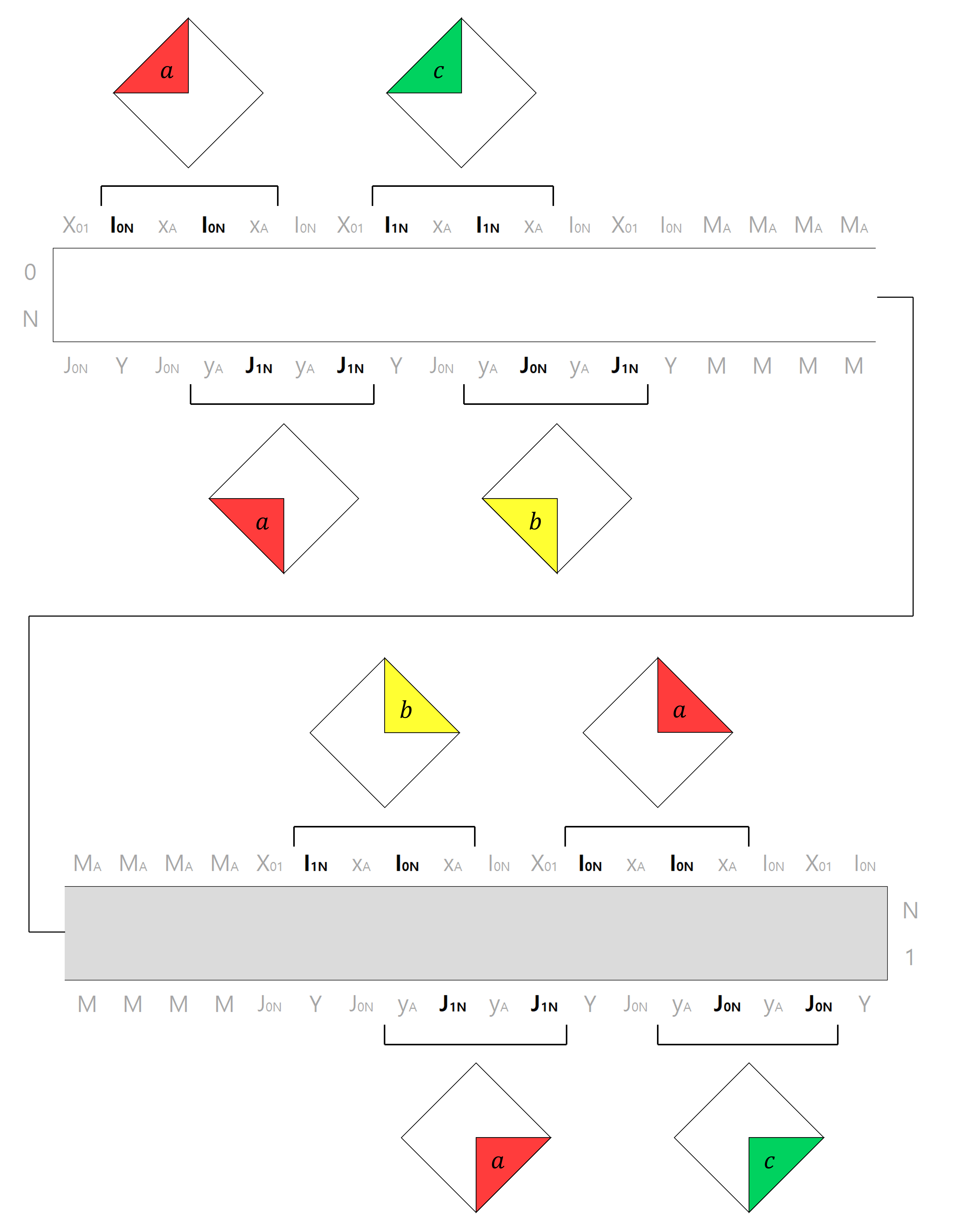}
\caption{An encoding of a set of 2 Wang tiles}
\label{fig:Figure18}
\end{figure}

Next, we explain the conditions for the information to be read.
The information labels encode the colors as their possibility to match with labels $0$ or $1$.
Thus, if an information label matches with a label other than $0$ or $1$, its information is discarded.
This occurs when two \emph{rods} of any type match along the long side or when a \emph{meat rod} and a \emph{filler rod} match.
The latter is shown in Figure $19$.
We can observe that every information label of the \emph{filler rods} matches with label $N$  of the \emph{meat rods}.

\begin{figure}[ht]
\centering
\includegraphics[width=0.9 \textwidth]{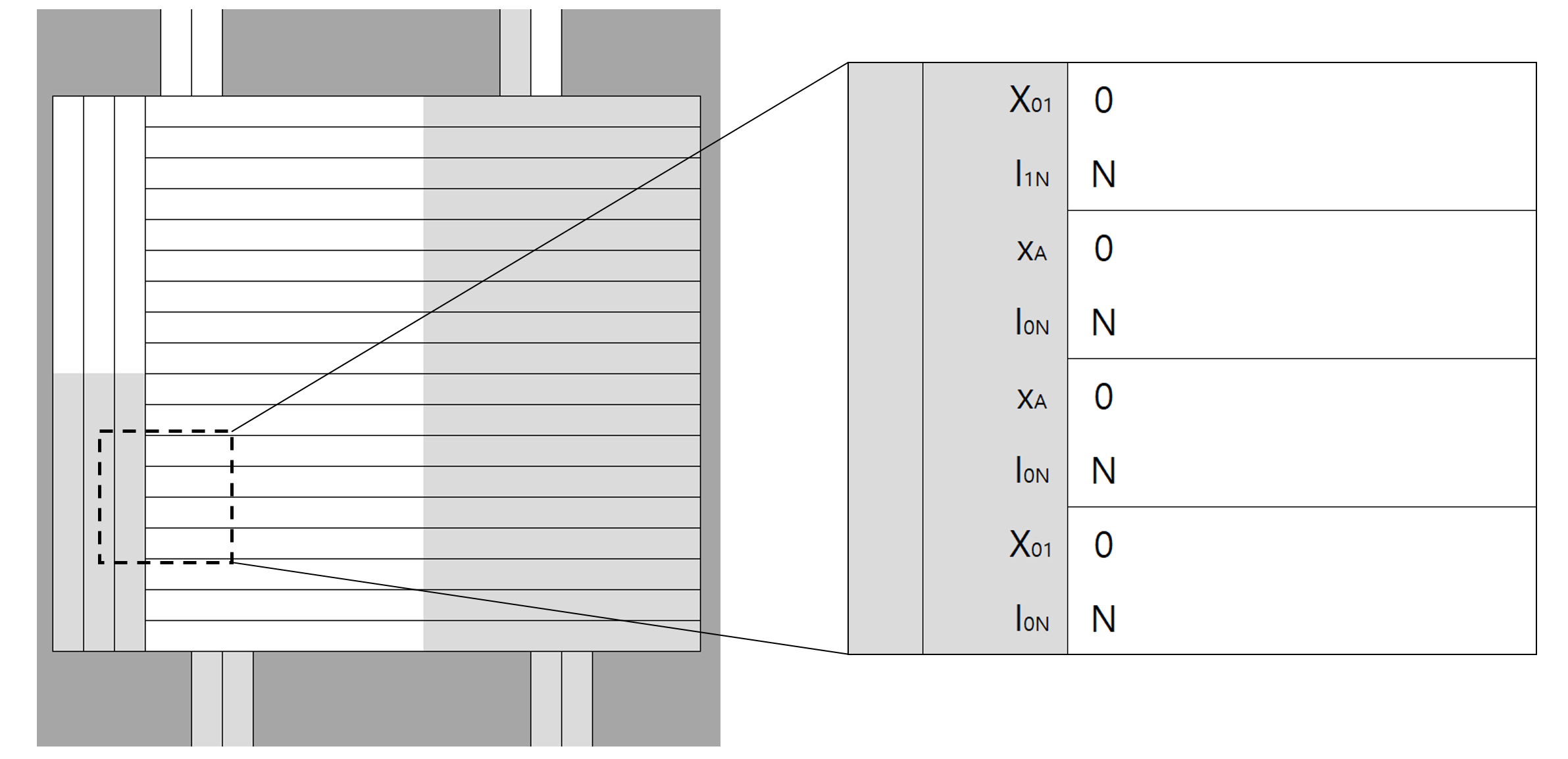}
\caption{An area of the tiling where no information is read}
\label{fig:Figure19}
\end{figure}

When \emph{wire rods} are placed next to a \emph{meat rod}, their labels $0$ and $1$ match with the information labels of the \emph{meat rods}(Figure 20).
Thus, the \emph{wire rods} can only be placed in one of the two orientations determined by the stored information.
For example, if the information label exposed to the \emph{wire rod} is $I_{0N}$ or $J_{0N}$,
the \emph{wire rod} must be oriented so that its label $0$ faces the \emph{meat rod}.

\begin{figure}[ht]
\centering
\includegraphics[width=0.9 \textwidth]{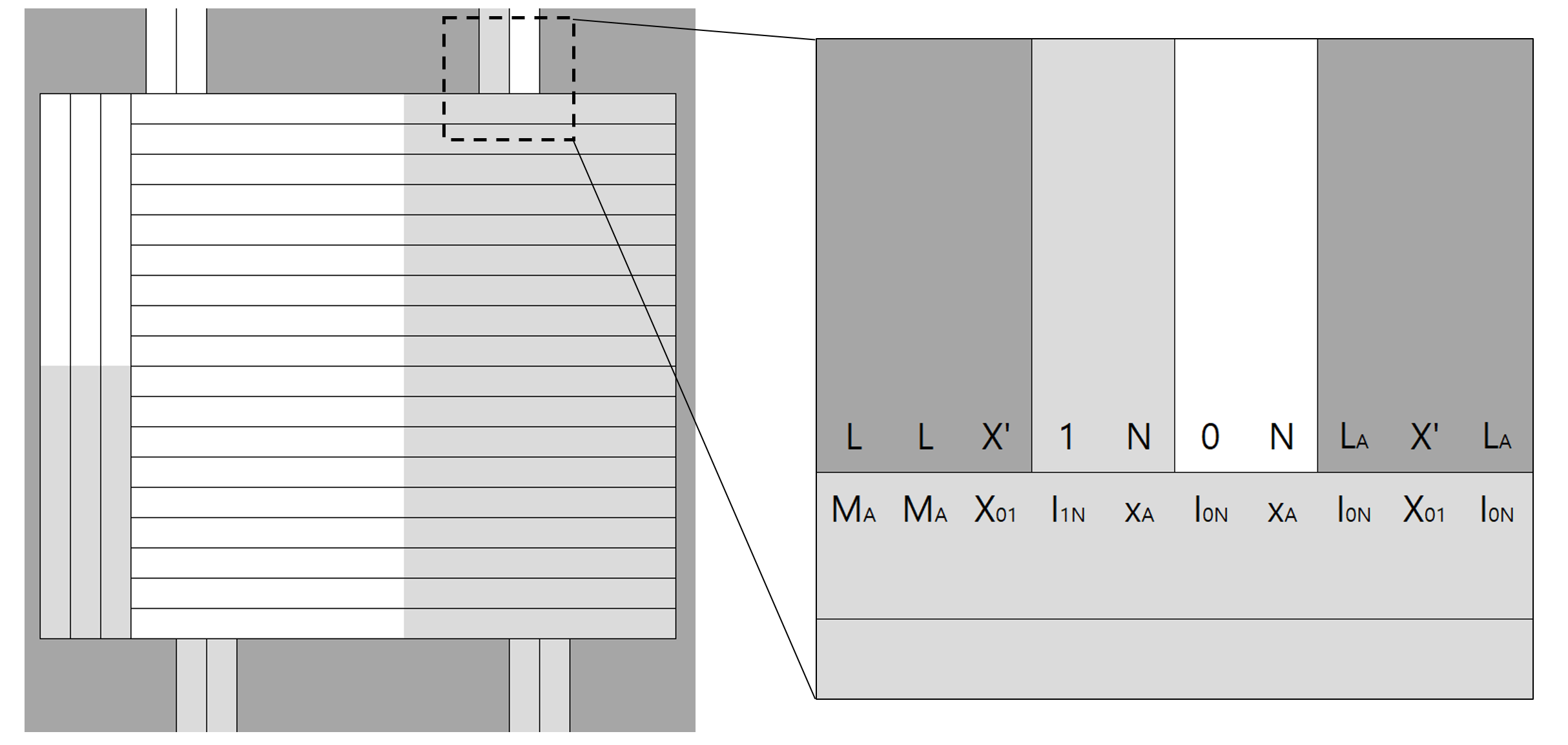}
\captionsetup{width=\textwidth}
\caption{\emph{Wire rods} reading information of the \emph{meat rod}}
\label{fig:Figure20}
\end{figure}

Finally, since the $2$ short sides of a \emph{wire rod} each contains labels $0$ and $1$, the $0$'s and the $1$’s in the information will be swapped.
For example, if one side of the $m-1$ \emph{wire rods} matched with $I_{1N} I_{0N} I_{0N}$ , the other side must match with $J_{0N} J_{1N} J_{1N}$.
Following the color encoding rule, this condition translates into $2$ edges of the same color being able to match.
This is shown in Figure $21$.

\begin{figure}[ht]
\centering
\includegraphics[width=0.9 \textwidth]{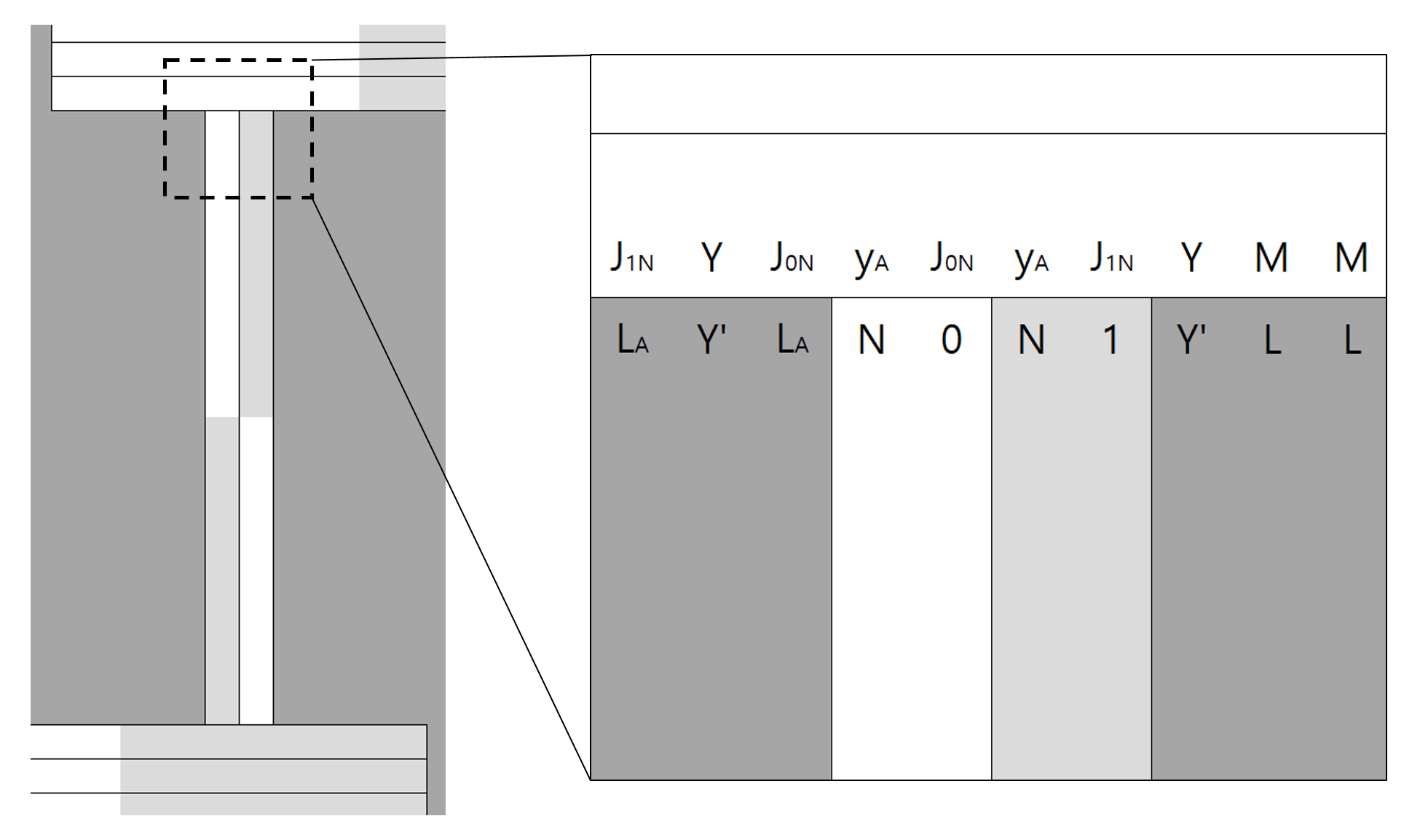}
\captionsetup{width=0.95 \textwidth}
\caption{The other end of the \emph{wire rods} shown in figure 20; the sections matching with either ends of these \emph{wire rods} encode the same color.}
\label{fig:Figure21}
\end{figure}

\subsection{Uniqueness of the tiling structure} 
We will finish the proof for Theorem 3.1 by showing that if \emph{teeth, rods} and \emph{blades} permit a tiling, they must form the structure in Figure 17.

\begin{lemma} 
It is impossible to tile the plane with only teeth and rods.
\end{lemma}

\begin{proof}
Assume that it is possible to tile the plane with only \emph{teeth} and \emph{rods}.
By Lemma 2.1, a \emph{rod} must be used.
Since we are not using any \emph{blades}, the label $M$ on the \emph{rod} can only match with labels $M$ or $M_A$.
In order for all labels $M$ to match with labels $M$ or $M_A$, there are two possible configurations for the second \emph{rod} as shown in Figure $22$.

\begin{figure}[ht]
\centering
\includegraphics[width=0.7 \textwidth]{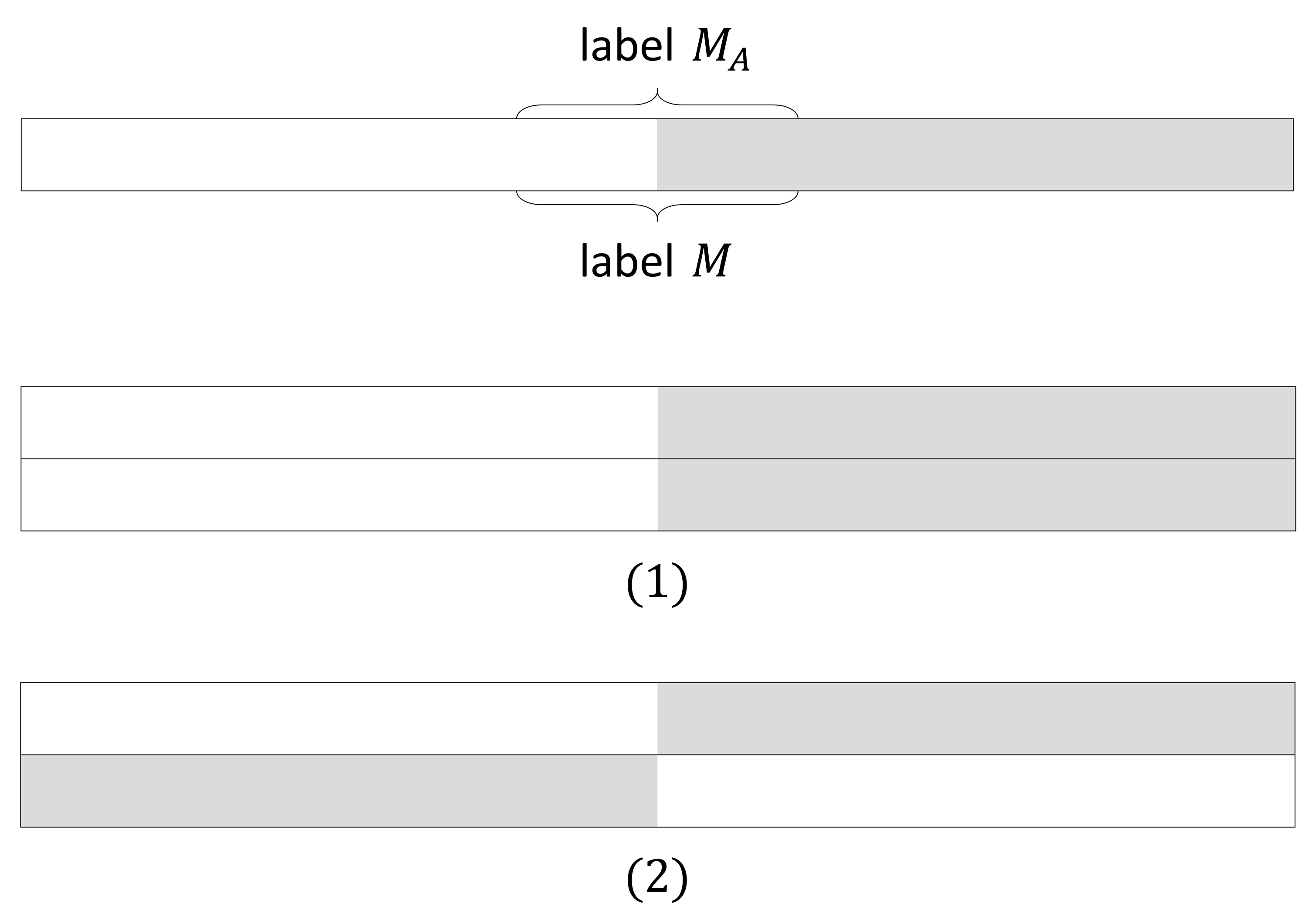}
\captionsetup{width=\textwidth}
\caption{Two possibilities for covering the labels $M$ }
\label{fig:Figure22}
\end{figure}

We can observe that the second configuration cannot be used since it forces the labels $Y$ to match with labels $J_{0N}$ or $J_{1N}$, which is invalid.
Placing the second \emph{rod} as the first configuration leaves another set of labels $M$ exposed.
Repeating the previous process, we can see that the \emph{rods} must continue infinitely in one direction(Figure $23$).

\begin{figure}[ht]
\centering
\includegraphics[width=0.7 \textwidth]{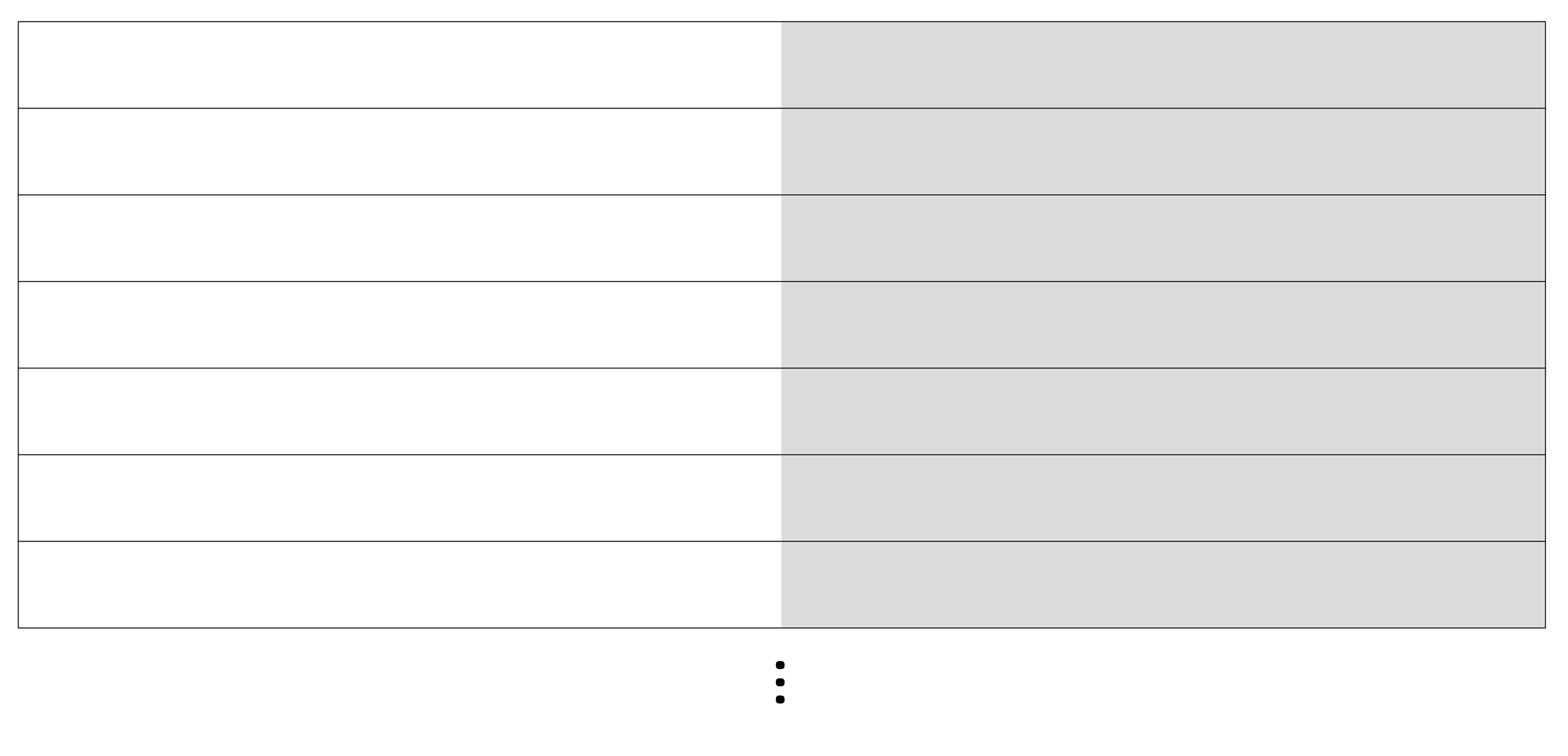}
\captionsetup{width=\textwidth}
\caption{The forced tiling of the \emph{rods}}
\label{fig:Figure23}
\end{figure}

Consider the short side of these \emph{rods}. Labels $0$, $1$ or $N$ cannot match with another label $0$, $1$ or $N$.
Therefore every short side of the \emph{rods} must match with the long side of a \emph{rod}.
In order to cover the short sides in Figure $23$, a \emph{rod} must be placed vertically along the short sides.
Regardless of where this \emph{rod} is placed, it creates a dent only consisting of short sides(Figure $24$).
This space cannot be filled, thus leading to a contradiction.
\end{proof}

\begin{figure}[ht]
\centering
\includegraphics[width=0.7 \textwidth]{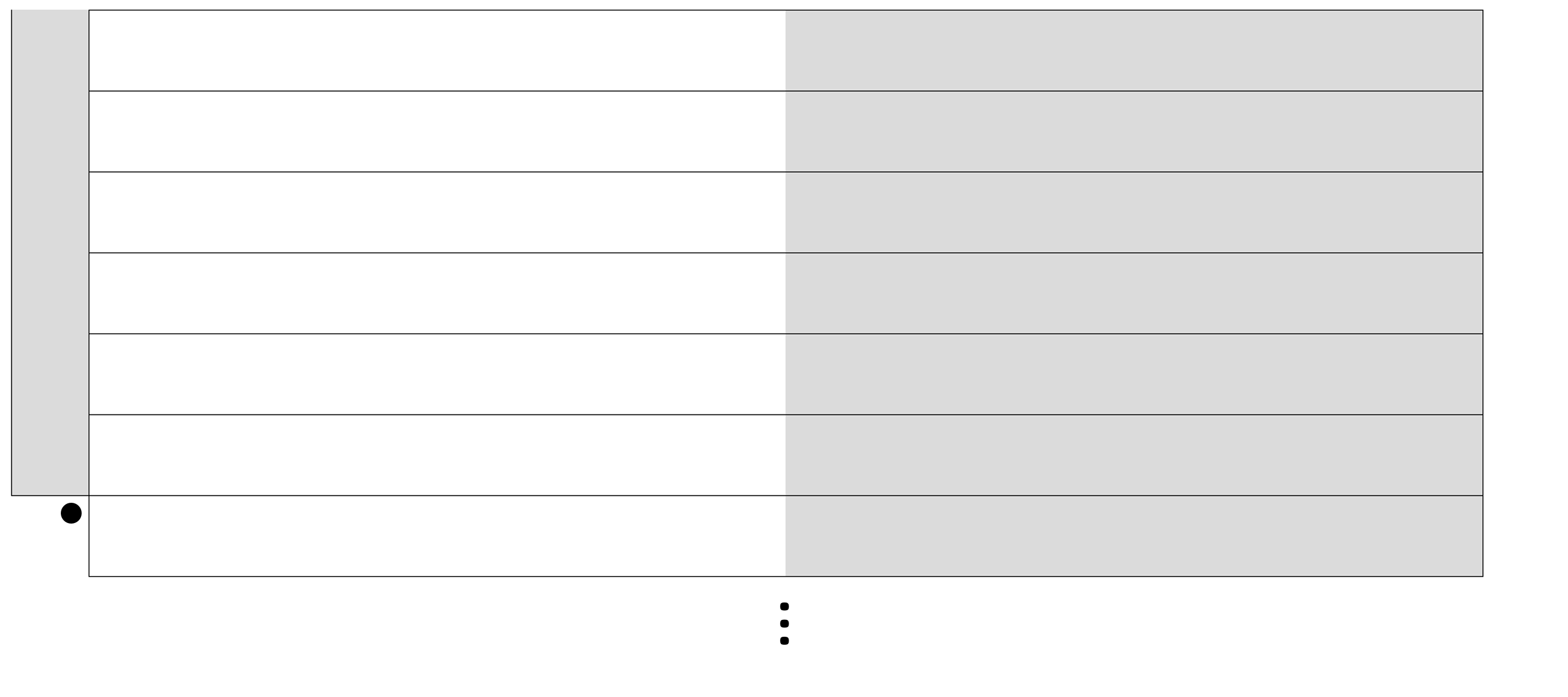}
\captionsetup{width=\textwidth}
\caption{The forced tiling of the \emph{rods} creating a dent that cannot be filled}
\label{fig:Figure24}
\end{figure}

By Lemma 3.2, a \emph{blade} must be used in the tiling.
Once a \emph{blade} is placed, we can determine the tiling around the \emph{blade} with the process below:

\begin{figure}[ht]
\centering
\includegraphics[width=0.8 \textwidth]{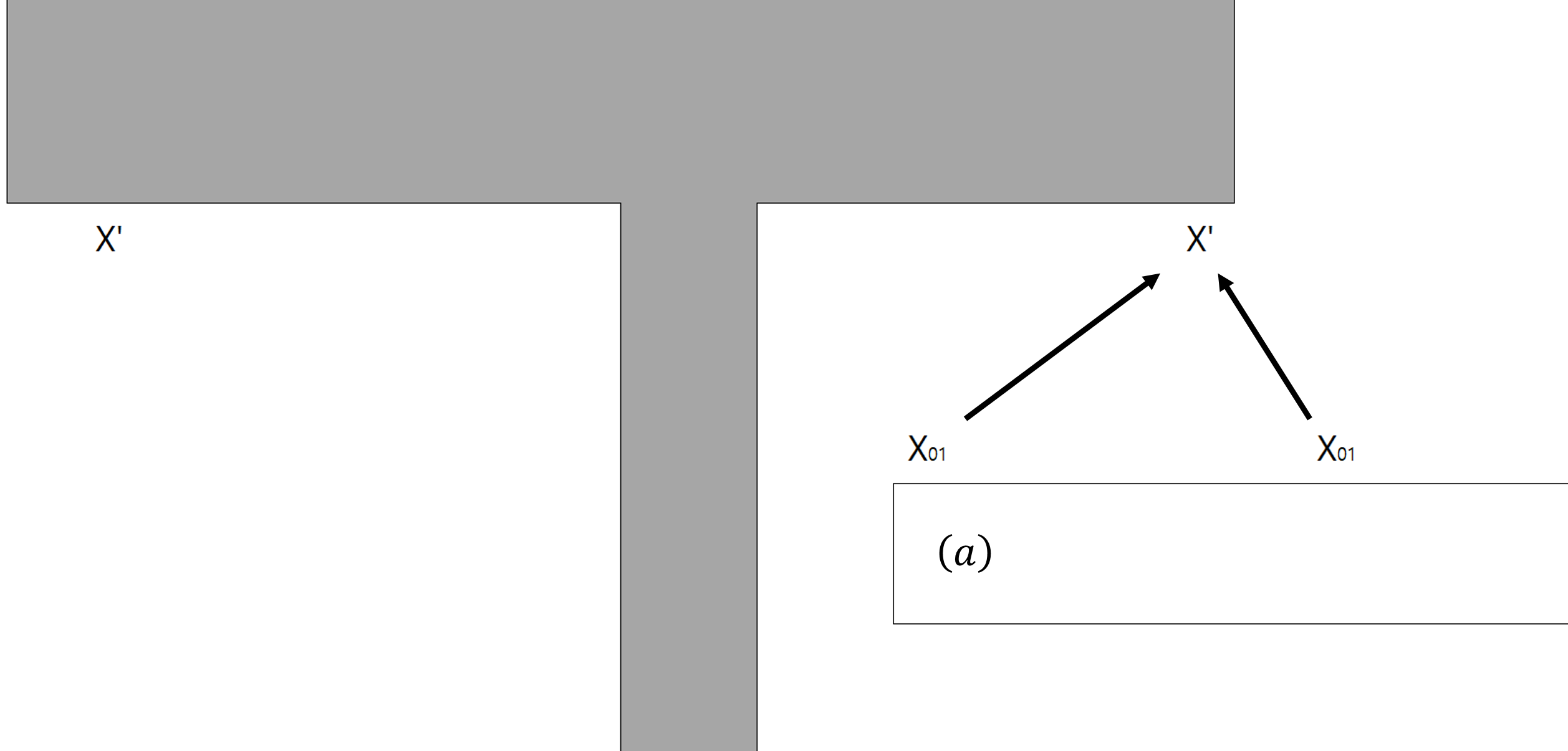}
\captionsetup{width=\textwidth}
\caption{Tiling possibilities around the label $X'$}
\label{fig:Figure25}
\end{figure}

\begin{itemize}
\setlength{\leftskip}{-3.5em}
\item[] \textbf{(a)} The label  $X'$ on the \emph{blade} can only match with one of the labels $X_{01}$ on a \emph{rod}.
Matching one of the labels $X_{01}$ with the label $X'$, we have chosen the Wang tile that this area will represent.
From this point, the position of every polyomino in this area is uniquely determined.
\end{itemize}

\begin{figure}[ht]
\centering
\includegraphics[width=0.7 \textwidth]{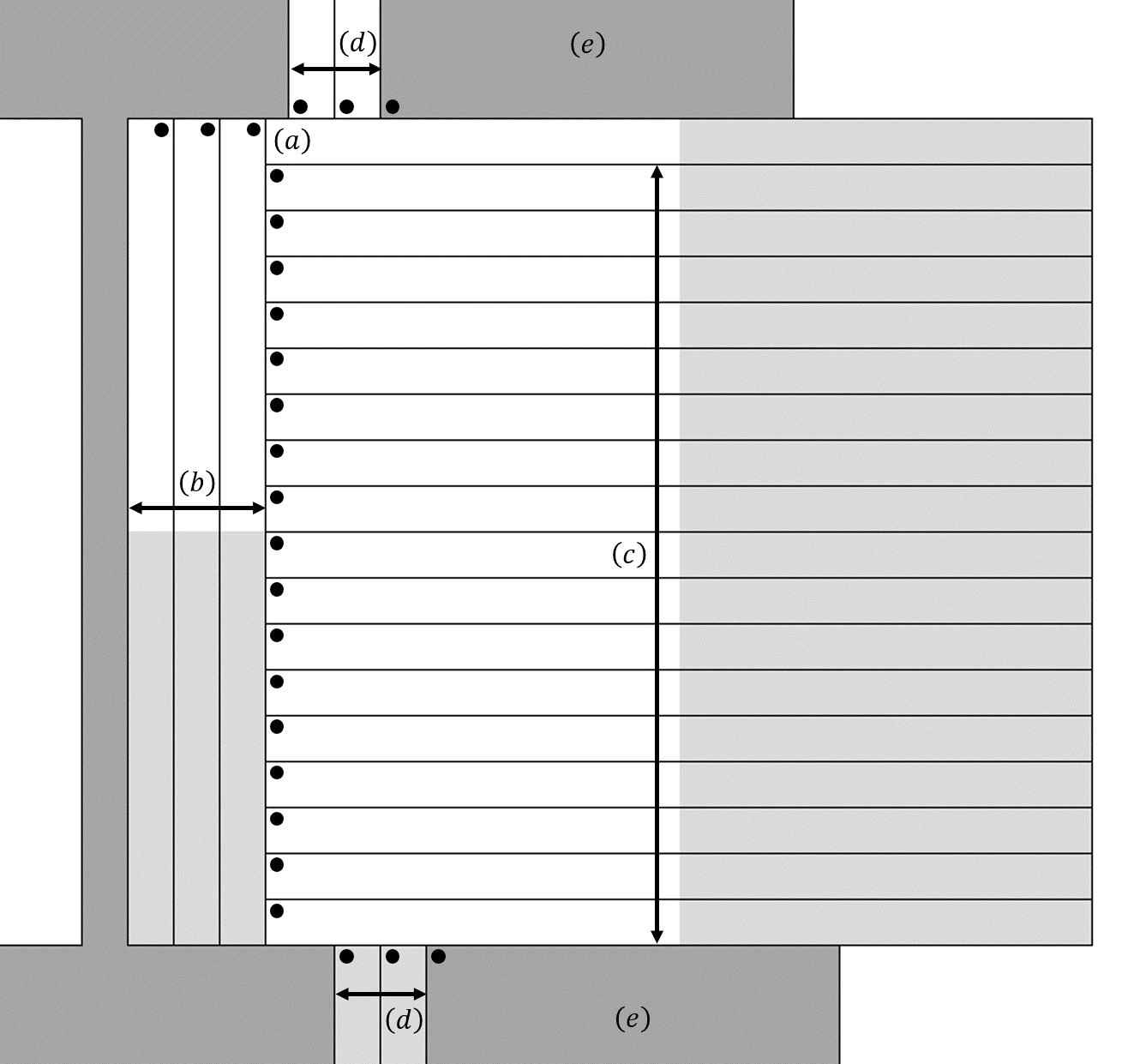}
\captionsetup{width=0.7 \textwidth}
\caption{Tiling possibilities around one side of a \emph{blade}; dots represent the space being analyzed in each step.}
\label{fig:Figure26}
\end{figure}

\begin{itemize}
\setlength{\leftskip}{-3.5em}

\item[] \textbf{(b)} Only vertical \emph{rods} can fit in the dent between \emph{rod} (a) and the \emph{blade}.
To match with the left side of (a), all vertical \emph{rods} are forced to an orientation where their label $Y$ faces the \emph{blade}.

\item[] \textbf{(c)} The remaining edges of the rightmost \emph{rod} of (b) must be covered by horizontal \emph{rods}.
This is because if a horizontal \emph{blade} is placed here,
the labels $X'$ force the \emph{blade} to leave an odd number of consecutive labels on the \emph{rod}, which cannot be covered.
A vertical \emph{blade} or a vertical \emph{rod} cannot fit.\\
To match with the bottom side of (a), all horizontal \emph{rods} are forced to an orientation where their label $X_{01}$ faces (a).

\item[] \textbf{(d)} The top edges of \emph{rod} (a) to the right of the \emph{blade} can only be filled by $(m-1)$ vertical \emph{rods}.
Let us show that every dot must be covered by a vertical \emph{rod}.\\
A horizontal \emph{rod} cannot cover these dots because it results in at least one label $M_A$ from \emph{rod} (a) forming an invalid match.\\
A horizontal \emph{blade} cannot cover the first dot because two \emph{blades} cannot match;\\
it also cannot cover the $(i+1)$th dot while a vertical \emph{rod} covers the $i$th dot since this matching corresponds to leaving $0$ labels on one side of the \emph{rod}, which is prevented by the labels $X'$.\\
A vertical \emph{blade} cannot be placed because the distance between labels $X_{01}$ on \emph{rod} (a) is greater than $2(m-1)$,
therefore labels $X'$ of two vertical \emph{blades} must also be more than $2(m-1)$ labels apart.\\
In the same way, the space below the bottommost \emph{rod} of (c) must be filled by $(m-1)$ \emph{rods}.

\item[] \textbf{(e)} If the upper vertical \emph{rods} in (d) continue past $(m-1)$ \emph{rods},
the label $N$ of the $m$th vertical \emph{rod} has to match with the label $X_{01}$ of \emph{rod} (a), which is invalid.
Therefore, only $(m-1)$ \emph{rods} can be placed, leaving the label $X_{01}$ not covered.
For the same reason as (d), a horizontal \emph{rod} or a horizontal \emph{blade} cannot be used to cover this label.
Therefore, the label $X_{01}$ must be covered by a vertical \emph{blade}.\\
In the same way, a vertical \emph{blade} must be placed below the bottommost \emph{rod} of (c).
\end{itemize}

We can also apply the process (a)$\textendash$(e) to the label $X'$ on the left of the \emph{blade} (Figure 25).
This determines the position of two more \emph{blades}.
We obtain the positions of all other polyominoes in the plane by repeatedly applying the process (a)$\textendash$(e) on the new \emph{blades}.
This shows that every tiling by the \emph{tooth, rod} and \emph{blade} corresponds to a tiling by their encoded Wang tiles.
Therefore the \emph{tooth, rod} and \emph{blade} can tile the plane if and only if the encoded set of Wang tiles can tile the plane.

\begin{theorem} 
The $5$-Polyomino translation tiling problem is undecidable.
\end{theorem}

\begin{proof}
In order to obtain the tiling structure in Theorem 3.1 without rotations, $1$ \emph{tooth}, $3$ \emph{rods} and $1$ \emph{blade} are required.
It is apparent that any tiling obtained from the translations of these $5$ polyominoes is also a tiling by the \emph{tooth, rod} and \emph{blade}, allowing rotations.
Thus the result in Section 3.4 holds, and these $5$ polyominoes can tile the plane if and only if the encoded set of Wang tiles can tile the plane. 
\end{proof}

\section{Conclusion} 
In this paper, we construct a set of polyominoes encoding an arbitrary set of Wang tiles to prove the undecidability of the $5$-Polyomino translation tiling problem.
It is worth noting that the $KL$ labeling technique is a major part of this construction. 
Transforming the shapes of dents on polyominoes using a polyomino much smaller than the dents is a new approach regarding polyomino constructions.
It is likely that $KL$ labeling can be utilized in other polyomino problems, as it is an effective way to force an arbitrary matching rule.

The result of this paper currently holds the record for the minimum number of shapes required to force aperiodicity or undecidability by translation. It will be interesting to investigate if a similar approach outside of polyominoes can further decrease the number of shapes required.

\end{document}